\documentclass{amsart}

\usepackage{amsfonts}
\usepackage{amsmath}
\usepackage{amssymb}
\usepackage{amsthm}
\usepackage{color}
\usepackage[colorlinks]{hyperref}
\usepackage{esint}
\usepackage{mathrsfs}

\theoremstyle{plain}
\newtheorem{theorem}{Theorem}[section]
\newtheorem{corollary}[theorem]{Corollary}
\newtheorem{lemma}[theorem]{Lemma}
\newtheorem{proposition}[theorem]{Proposition}

\theoremstyle{definition}
\newtheorem{definition}[theorem]{Definition}

\theoremstyle{remark}

\numberwithin{equation}{section}

\newcommand{\R}{\mathbb{R}}
\newcommand{\Rn}{\mathbb{R}^n}
\newcommand{\ve}{\vert}
\newcommand{\lv}{\left\vert}
\newcommand{\rv}{\right\vert}
\newcommand{\Ve}{\Vert}

\newcommand{\lb}{\left\lbrace}
\newcommand{\rb}{\right\rbrace}

\newcommand{\us}{\underline{\sigma}}
\newcommand{\os}{\overline{\sigma}}
\newcommand{\uC}{\underline{C}}
\newcommand{\oC}{\overline{C}}

\begin{document}

\title[Integro-differential operators with kernels of variable orders]
{Regularity for fully nonlinear integro-differential operators with kernels of variable orders}

\author{Minhyun Kim}
\address{Department of Mathematical Sciences, Seoul, Korea}
\email{201421187@snu.ac.kr}
\author{Ki-Ahm Lee}
\address{Department of Mathematical Sciences, Seoul, Korea \& Korea Institute for Advanced Study, Seoul, Korea}
\email{kiahm@snu.ac.kr}


\begin{abstract}
We consider fully nonlinear elliptic integro-differential operators with kernels of variable orders, which generalize the integro-differential operators of the fractional Laplacian type in \cite{CS}. Since the order of differentiability of the kernel is not characterized by a single number, we use the constant
\begin{align*}
C_\varphi = \left( \int_{\Rn} \frac{1-\cos y_1}{\ve y \ve^n \varphi (\ve y \ve)} \, dy \right)^{-1}
\end{align*}
instead of $2-\sigma$, where $\varphi$ satisfies a weak scaling condition. We obtain the uniform Harnack inequality and H\"older estimates of viscosity solutions to the nonlinear integro-differential equations.
\end{abstract}

\maketitle

\section{Introduction}

In this paper we consider fully nonlinear elliptic integro-differential operators. By the L\'evy Khintchine formula, the generator of an $n$-dimensional pure jump process is given by
\begin{align} \label{e:L general}
Lu(x) = \int_{\Rn} \left( u(x+y) - u(x) - \nabla u(x) \cdot y \chi_{B_1} (y) \right) \, d\mu(y),
\end{align}
where $\mu$ is a measure such that $\int_{\Rn} \ve y \ve^2 / (1 + \ve y \ve^2) \, d\mu(y) < \infty$. Note that the value of $Lu(x)$ is well-defined as long as $u$ is bounded in $\Rn$ and $C^{1,1}$ in a neighborhood of $x$. Since the operators are given in too much generality, we restrict ourselves to the operators given by symmetric kernels $K$. In this case, the operator \eqref{e:L general} can be written as
\begin{align} \label{e:L symmetric}
Lu(x) = \int_{\Rn} (u(x+y) + u(x-y) - 2u(x)) K(y) \, dy,
\end{align}
and the kernel $K$ satisfies
\begin{align} \label{a:K}
\int_{\Rn} \ve y \ve^2 / (1+\ve y \ve^2) K(y) \, dy < \infty.
\end{align}
For the notational convenience we write $\delta(u, x, y) = u(x+y) + u(x-y) - 2u(x)$ in the sequel. Nonlinear integro-differential operators such as
\begin{align*}
Iu(x) = \sup_\alpha L_\alpha u(x) \quad \text{or} \quad Iu(x) = \inf_\beta \sup_\alpha L_{\alpha\beta} u(x)
\end{align*}
arise in the stochastic control theory and the game theory. A characteristic property of these operators is that
\begin{align*}
\inf_{\alpha\beta} L_{\alpha\beta} v(x) \leq I(u+v)(x) - Iu(x) \leq \sup_{\alpha\beta} L_{\alpha\beta} v(x).
\end{align*}
Caffarelli and Silverstre \cite{CS} introduced the concept of ellipticity for more general nonlinear operators $I$: for a class of linear integro-differential operators $\mathcal{L}$ it holds that
\begin{align*}
\mathcal{M}^-_{\mathcal{L}} (u-v)(x) \leq Iu(x) - Iv(x) \leq \mathcal{M}^+_{\mathcal{L}} (u-v)(x),
\end{align*}
where $\mathcal{M}^+_{\mathcal{L}}$ and $\mathcal{M}^-_{\mathcal{L}}$ are a maximal and a minimal operator with respect to $\mathcal{L}$, defined by
\begin{align*}
\mathcal{M}^+_{\mathcal{L}} u(x) = \sup_{L \in \mathcal{L}} Lu(x) \quad\text{and} \quad \mathcal{M}^-_{\mathcal{L}} u(x) = \inf_{L \in \mathcal{L}} Lu(x),
\end{align*}
respectively. See \cite{CC} for elliptic second-order differential operators. We adopt this concept and will give a precise definition in Section \ref{s:3 VS}.

Caffarelli and Silverstre \cite{CS} considered fully nonlinear integro-differential operators with kernels comparable to those of fractional Laplacian to obtain regularity results. That is, they considered the class of operators of the form \eqref{e:L symmetric} with
\begin{align*}
(2-\sigma) \frac{\lambda}{\ve y \ve^{n+\sigma}} \leq K(y) \leq (2-\sigma) \frac{\Lambda}{\ve y \ve^{n+\sigma}},
\end{align*}
where $0 < \sigma < 2$. They obtained regularity estimates that remain uniform as the order of the equation $\sigma$ approaches 2 and therefore made the theory of integro-differential equations and elliptic differential equations appear somewhat unified. More generally, in \cite{KKL} the authors generalized these results to fully nonlinear integro-differential operators with regularly varying kernels. More precisely, they considered the class of operators of the form \eqref{e:L symmetric} with
\begin{align*}
(2-\sigma) \lambda \frac{l(\ve y \ve)}{\ve y \ve^n} \leq K(y) \leq (2-\sigma) \Lambda \frac{l(\ve y \ve)}{\ve y \ve^n},
\end{align*}
where $l : (0, \infty) \rightarrow (0, \infty)$ is a locally bounded, regularly varying function at zero with index $-\sigma$. In both cases, the constant $2-\sigma$ plays a very important role in uniform regularity estimates. They used the constant $2-\sigma$ instead of the constant in the fractional Laplacian 
\begin{align*}
C(n, \sigma) = \left( \int_{\Rn} \frac{1-\cos y_1}{\ve y \ve^{n+\sigma}} \, dy \right)^{-1} = \frac{2^\sigma \Gamma(\frac{n+\sigma}{2})}{\pi^{n/2} \ve \Gamma (-\frac{\sigma}{2}) \ve}
\end{align*}
because two constants $2-\sigma$ and $C(n, \sigma)$ have the same asymptotic behavior as $\sigma$ approaches 2 and they focused on regularity estimates which remain uniform as $\sigma$ approaches 2.

In this paper we will consider kernels of variable orders. In this case the order of the kernel cannot be characterized in a single number. This implies that we need to consider the constant which contains all information of the kernel to generalize the results of \cite{CS}. We will define this constant in Section \ref{s:IO}

\subsection{Integro-differential Operators} \label{s:IO}

In order to obtain regularity results, we need to impose some assumptions on the kernel $K$. Throughout this paper, we will assume that the kernel $K$ satisfies
\begin{align} \label{a:K comp}
C_\varphi \frac{\lambda}{\ve y \ve^n \varphi(\ve y \ve)} \leq K(y) \leq C_\varphi \frac{\Lambda}{\ve y \ve^n \varphi(\ve y \ve)}
\end{align}
for some constants $0 < \lambda \leq \Lambda < \infty$, where a function $\varphi : (0, \infty) \rightarrow (0, \infty)$ and a constant $C_\varphi$ will be defined below.

We first assume that the function $\varphi$ satisfies a {\it weak scaling condition} with constants $a \geq 1$ and $0 < \us \leq \os < 2$, i.e.,
\begin{align} \label{a:wsc}
a^{-1} \left( \frac{R}{r} \right)^{\us} \leq \frac{\varphi(R)}{\varphi(r)} \leq a \left( \frac{R}{r} \right)^{\os} \quad \text{for all} ~ 0 < r \leq R < \infty.
\end{align}
The simplest example of this function is $\varphi(r) = r^\sigma$ with $\sigma \in (0,2)$, which corresponds to the fractional Laplacian. However, more general functions such as $\varphi(r) = r^{\us} + r^{\os}, \varphi(r) = r^{\us} (\log (1+r^{-2}))^{-(2-\os)/2}$, and $\varphi(r) = r^{\os} (\log(1+r^{-2}))^{\us/2}$ are covered.

We next observe that if we take the Fourier transform to the operator
\begin{align*}
L_0 u(x) = \int_{\Rn} \frac{u(x+y) + u(x-y) - 2u(x)}{\ve y \ve^n \varphi(\ve y \ve)} \, dy,
\end{align*}
then
\begin{align*}
- \mathscr{F}(L_0 u)(\xi) 
&= - \int_{\Rn} \frac{\mathscr{F} (u(\cdot + y) + u(\cdot - y) - 2u(\cdot))(\xi)}{\ve y \ve^n \varphi(\ve y \ve)} \, dy \\
&= - \int_{\Rn} \frac{e^{i \xi \cdot y} + e^{-i \xi \cdot y} - 2}{\ve y \ve^n \varphi(\ve y \ve)} \, dy  (\mathscr{F}u)(\xi) \\
&= 2\int_{\Rn} \frac{1-\cos(\xi \cdot y)}{\ve y \ve^n \varphi(\ve y \ve)} \, dy (\mathscr{F}u)(\xi).
\end{align*}
Since the function
\begin{align*}
\xi \mapsto \int_{\Rn} \frac{1-\cos(\xi \cdot y)}{\ve y \ve^n \varphi(\ve y \ve)} \, dy
\end{align*}
is rotationally symmetric, we have 
\begin{align} \label{e:symbol}
- \mathscr{F}(L_0 u)(\xi) = 2\int_{\Rn} \frac{1-\cos(\ve \xi \ve y_1)}{\ve y \ve^n \varphi(\ve y \ve)} \, dy (\mathscr{F}u)(\xi).
\end{align}
Note that when $\varphi(r) = r^\sigma$ the integral in \eqref{e:symbol} can be represented as
\begin{align*}
\int_{\Rn} \frac{1-\cos(\ve \xi \ve y_1)}{\ve y \ve^{n+\sigma}} \, dy = \int_{\Rn} \frac{1-\cos y_1}{\ve y / \ve \xi \ve \ve^{n+\sigma}} \frac{dy}{\ve \xi \ve^n} = C(n, \sigma)^{-1} \ve \xi \ve^\sigma,
\end{align*}
and hence the fractional Laplacian is defined with the constant $C(n, \sigma)$ as
\begin{align*}
-(-\Delta)^{\sigma/2} u(x) = \frac{1}{2} C(n, \sigma) \int_{\Rn} \frac{\delta(u, x, y)}{\ve y \ve^{n+\sigma}} \, dy.
\end{align*}
Thus, in the general case, it is natural to define
\begin{align*}
C_\varphi = \left( \int_{\Rn} \frac{1 - \cos y_1}{\ve y \ve^n \varphi(\ve y \ve)} \, dy \right)^{-1}
\end{align*}
as a normalizing constant. Then the operator $\frac{1}{2} C_\varphi L_0$ generalizes the fractional Laplacian $-(-\Delta)^{\sigma/2}$. In Section \ref{s:2 Asymptotics}, we will prove asymptotic properties of the constant $C_\varphi$ and the operator $\frac{1}{2} C_\varphi L_0$.

\subsection{Main Results}

In this paper, we are concerned with the nonlinear integro-differential operator
\begin{align} \label{e:I}
Iu := \inf_{\beta} \sup_{\alpha} L_{\alpha\beta} u, \quad L_{\alpha\beta} \in \mathcal{L}_0,
\end{align}
where $\mathcal{L}_0$ denotes the class of linear integro-differential operators of the form \eqref{e:L symmetric} with symmetric kernels $K$ satisfying \eqref{a:K} and \eqref{a:K comp}.

We define functions $\uC, \oC : (0, \infty) \rightarrow \R$ by
\begin{align*}
\uC (R) := \uC_\varphi (R) := \int_0^R \frac{r}{\varphi(r)} \, dr \quad \text{and} \quad \oC (R) := \oC_\varphi (R) := \int_R^\infty \frac{1}{r\varphi(r)} \, dr.
\end{align*}
They correspond to $\frac{R^{2-\sigma}}{2-\sigma}$ and $\frac{R^{-\sigma}}{\sigma}$ for the case of fractional Laplacian, respectively. We will denote by $\uC = \uC (1)$ and $\oC = \oC (1)$.

Now we present our main results which generalize the uniform regularity results in \cite{CS}. Throughout this paper we denote $B_R := B_R(0)$ for $R > 0$.

\begin{theorem} [Harnack inequality] \label{t:Harnack}
Let $\sigma_0 \in (0,2)$ and assume $\us \geq \sigma_0$. Let $u \in C(B_{2R})$ be a nonnegative function in $\Rn$ such that
\begin{align*}
\mathcal{M}^-_{\mathcal{L}_0} u \leq C_0 \quad \text{and} \quad \mathcal{M}^+_{\mathcal{L}_0} u \geq - C_0 \quad \text{in} ~ B_{2R}
\end{align*}
in the viscosity sense. Then there exists a uniform constant $C > 0$, depending only on $n, \lambda, \Lambda, a$, and $\sigma_0$, such that
\begin{align} \label{e:Harnack}
\sup_{B_R} u \leq C \left( \inf_{B_R} u + C_0 \frac{(\uC + \oC)R^2}{\uC(R)} \right).
\end{align}
\end{theorem}

\begin{theorem} [H\"older regularity] \label{t:Holder}
Let $\sigma_0 \in (0,2)$ and assume $\us \geq \sigma_0$. Let $u \in C(B_{2R})$ be a function in $\Rn$ such that
\begin{align*}
\mathcal{M}^-_{\mathcal{L}_0} u \leq C_0 \quad \text{and} \quad \mathcal{M}^+_{\mathcal{L}_0} u \geq - C_0 \quad \text{in} ~ B_{2R}
\end{align*}
in the viscosity sense. Then $u \in C^\alpha(B_R)$ and
\begin{align} \label{e:Holder}
R^\alpha [u]_{C^\alpha(B_R)} \leq C \left( \Ve u \Ve_{L^\infty(\Rn)} + C_0 \frac{(\uC + \oC)R^2}{\uC(R)} \right)
\end{align}
for some uniform constants $\alpha > 0$ and $C > 0$ which depend only on $n, \lambda, \Lambda, a$, and $\sigma_0$.
\end{theorem}

It is important to note that in the regularity estimates \eqref{e:Harnack} and \eqref{e:Holder} the constants are independent of $\us$ and $\os$, but the term $\frac{(\uC + \oC)R^2}{\uC(R)}$ in the right-hand side of \eqref{e:Harnack} and \eqref{e:Holder} still depends on $\us$ and $\os$. For the fractional Laplacian case this term corresponds to $\frac{2}{\sigma} R^\sigma$ and it can be further estimated as
\begin{align*}
\frac{2}{\sigma} R^\sigma \leq \frac{2}{\sigma_0} R^\sigma.
\end{align*}
In our case, we can also estimate the term $\frac{(\uC + \oC)R^2}{\uC(R)}$ using Lemma \ref{l:bound} and Lemma \ref{l:wsc} as
\begin{align} \label{e:further est}
\frac{(\uC + \oC)R^2}{\uC(R)}
&\leq C(n, a) \frac{(\uC(R) + \oC(R))R^2}{\uC(R)} \leq C(n, a) \left( R^2 + \frac{2a^2}{\sigma_0} \right),
\end{align}
which is independent of $\us$ and $\os$. Notice that it has the same blow up rate with the fractional Laplacian case. Nevertheless, we leave \eqref{e:Harnack} and \eqref{e:Holder} as they are because the estimate \eqref{e:further est} has a different scale with respect to $R$.

This paper is organized as follows. In Section \ref{s:2 Asymptotics} we study asymptotic properties of the constant $C_\varphi$ and the operator $\frac{1}{2}C_\varphi L_0$, which play crucial roles in the forthcoming regularity results. Some bounds for the constant $C_\varphi$ are also given in this section. Section \ref{s:3 VS} is devoted to the definitions of viscosity solutions and the notion of ellipticity for nonlinear integro-differential operators. In Section \ref{s:4.1 ABP} we prove the ABP estimates, which is the main ingredient in the proof of Harnack inequality. We construct a barrier function in Section \ref{s:4.2 Barrier} and then use this function and ABP estimates to provide the measure estimates of super-level sets of the viscosity subsolutions to elliptic integro-differential equations in Section \ref{s:4.3 decay est}. We establish the Harnack inequality and H\"older estimates of viscosity solutions to elliptic integro-differential equations in Section \ref{s:4.4 Harnack} and \ref{s:4.5 Holder}.

\section{Asymptotics of the Constant \texorpdfstring{$C_\varphi$}{}} \label{s:2 Asymptotics}

It is well-known that the constant $C(n, \sigma)$ for the fractional Laplacian has the following asymptotic properties:
\begin{align} \label{e:asymp FL}
\lim_{\sigma \rightarrow 2^-} \frac{C(n, \sigma)}{2-\sigma} = \frac{2}{\omega_n} \quad \text{and} \quad \lim_{\sigma \rightarrow 0^+} \frac{C(n, \sigma)}{\sigma} = \frac{1}{n\omega_n},
\end{align}
where $\omega_n$ denotes the volume of the $n$-dimensional unit ball, and that the fractional Laplacian $(-\Delta)^{\sigma/2}$ has the following properties:
\begin{align} \label{e:limit FL}
\lim_{\sigma \rightarrow 2^-} (-\Delta)^{\sigma/2} u = -\Delta u \quad \text{and} \quad \lim_{\sigma \rightarrow 0^+} (-\Delta)^{\sigma/2} u = u.
\end{align}
See \cite{DNPV} for the proofs. In this section, we prove the analogues of \eqref{e:asymp FL} and \eqref{e:limit FL}, which will imply that the constant $C_\varphi$ generalizes the constant of the fractional Laplacian $C(n, \sigma)$.

To state the analogues of \eqref{e:asymp FL} and \eqref{e:limit FL}, we must consider a sequence of operators
\begin{align*}
L_k u(x) = \frac{1}{2} C_{\varphi_k} \int_{\Rn} \frac{\delta(u, x, y)}{\ve y \ve^n \varphi_k(\ve y \ve)} \, dy,
\end{align*}
where functions $\varphi_k : (0, \infty) \rightarrow (0, \infty)$ satisfy weak scaling conditions \eqref{a:wsc} with constants $a_k \geq 1$ and $0 < \us_k \leq \os_k < 2$. We will assume that
\begin{align} \label{a:a_k}
\lim_{k \rightarrow \infty} a_k = 1
\end{align}
throughout this section. The following lemma and proposition correspond to \eqref{e:asymp FL} and \eqref{e:limit FL}, respectively. Recall that $\uC(R) = \frac{R^{2-\sigma}}{2-\sigma}$ and $\oC(R) = \frac{R^{-\sigma}}{\sigma}$ in the fractional Laplacian case.
\begin{lemma} \label{l:asymp}
Assume that \eqref{a:a_k} holds. If $\displaystyle \lim_{k \rightarrow \infty} \us_k = \lim_{k \rightarrow \infty} \os_k = 2$, then
\begin{align} \label{e:asymp2}
\lim_{k \rightarrow \infty} C_{\varphi_k} \uC_{\varphi_k} (R) = \frac{2}{\omega_n},
\end{align}
and if $\displaystyle \lim_{k \rightarrow \infty} \us_k = \lim_{k \rightarrow \infty} \os_k = 0$, then
\begin{align} \label{e:asymp0}
\lim_{k \rightarrow \infty} C_{\varphi_k} \oC_{\varphi_k} (R) = \frac{1}{n \omega_n}.
\end{align}
\end{lemma}

\begin{proposition} \label{p:lim}
Assume that \eqref{a:a_k} holds, and let $u \in C_c^\infty (\Rn)$. If $\displaystyle \lim_{k \rightarrow \infty} \us_k = \lim_{k \rightarrow \infty} \os_k = 2$, then
\begin{align*}
\lim_{k \rightarrow \infty} -L_k u = -\Delta u,
\end{align*}
and if $\displaystyle \lim_{k \rightarrow \infty} \us_k = \lim_{k \rightarrow \infty} \os_k = 0$, then
\begin{align*}
\lim_{k \rightarrow \infty} -L_k u = u.
\end{align*}
\end{proposition}

The following estimates for the functions $\uC_\varphi(R)$ and $\oC_\varphi(R)$ will be used frequently in the sequel.

\begin{lemma} \label{l:wsc}
It holds that
\begin{align} \label{e:wsc 0}
\frac{1}{a(2-\us)} \frac{R^2}{\varphi(R)} \leq \uC_{\varphi} (R) \leq \frac{a}{2-\os} \frac{R^2}{\varphi(R)},
\end{align}
and that
\begin{align} \label{e:wsc infty}
\frac{1}{a\os} \frac{1}{\varphi(R)} \leq \oC_{\varphi} (R) \leq \frac{a}{\us} \frac{1}{\varphi(R)}.
\end{align}
Moreover, for $t \in (0,1)$ it holds that
\begin{align} \label{e:wsc ratio}
\frac{\uC_\varphi(R)}{\uC_\varphi(tR)} \leq 1+ a^2 t^{-2+\us}.
\end{align}
\end{lemma}

\begin{proof}
Using the weak scaling condition \eqref{a:wsc} we see that
\begin{align*}
\uC_\varphi (R) \geq \int_0^R \frac{1}{a} \frac{r}{\varphi(R)} \left( \frac{R}{r} \right)^{\us} \, dr = \frac{1}{a(2-\us)} \frac{R^2}{\varphi(R)},
\end{align*}
which is the first inequality in \eqref{e:wsc 0}. The second inequality in \eqref{e:wsc 0} and the inequalities in \eqref{e:wsc infty} can be proved in the same manner. The last inequality follows from \eqref{a:wsc} and \eqref{e:wsc 0} that
\begin{align*}
\frac{\uC_\varphi (R)}{\uC_\varphi (tR)}
&= 1 + \frac{1}{\uC_\varphi(tR)} \int_{tR}^R \frac{r}{\varphi(r)} \, dr \leq 1 + a(2-\us) \frac{\varphi(tR)}{(tR)^2} \int_{tR}^R a \frac{(tR)^{\us}}{\varphi(tR)} r^{1-\us} \, dr \\
&\leq 1+ a^2 (tR)^{-2+\us} \left( R^{2-\us} - (tR)^{2-\us} \right) \leq 1 + a^2 t^{-2+\us}.
\end{align*}
\end{proof}

Next we will prove Lemma \ref{l:asymp} using Lemma \ref{l:wsc} and the fact that the constant $C_\varphi$ can be represented by
\begin{align*}
C_\varphi^{-1}
&= \int_{\R} \int_{\R^{n-1}} \frac{1-\cos y_1}{\ve y_1 \ve^n \left( 1 + \frac{\ve y' \ve^2}{\ve y_1 \ve^2} \right)^{n/2} \varphi \left( \ve y_1 \ve \left( 1 + \frac{\ve y' \ve^2}{\ve y_1 \ve^2} \right)^{1/2} \right)} \, dy' dy_1 \\
&= \int_{\R} \int_{\R^{n-1}} \frac{1-\cos y_1}{\ve y_1 \ve (1+\ve y' \ve^2)^{n/2} \varphi(\ve y_1 \ve (1+\ve y' \ve^2)^{1/2})} \, dy' dy_1 \\
&= \int_{\R^{n-1}} \int_{\R} \frac{1-\cos \frac{r}{\zeta}}{\ve r \ve \varphi(\ve r \ve)} \, dr \frac{dy'}{\zeta^n},
\end{align*}
where $\zeta = \zeta(y') = (1+\ve y' \ve^2)^{1/2}$.

\begin{proof} [Proof of Lemma \ref{l:asymp}]
We first assume that
\begin{align} \label{a:2}
\lim_{k \rightarrow \infty} \us_k = \lim_{k \rightarrow \infty} \os_k = 2.
\end{align}
We use the inequality \eqref{e:wsc infty} to compute
\begin{align*}
0 &\leq \int_{\ve r \ve \geq R} \frac{1-\cos \frac{r}{\zeta}}{\ve r \ve \varphi_k (\ve r \ve)} \, dr \leq \int_R^\infty \frac{4}{r \varphi_k(r)} \, dr \leq \frac{4a_k}{\us_k \varphi_k(R)}.
\end{align*}
Using the assumptions \eqref{a:a_k}, \eqref{a:2}, and the inequality \eqref{e:wsc 0}, we obtain
\begin{align*}
0 \leq \lim_{k\rightarrow \infty} \frac{1}{\uC_{\varphi_k}(R)} \int_{\ve r \ve \geq R} \frac{1-\cos \frac{r}{\zeta}}{\ve r \ve \varphi_k (\ve r \ve)} \, dr \leq \lim_{k \rightarrow \infty} 4a_k^2 \frac{2-\us_k}{\us_k} R^{-2} = 0.
\end{align*}
On the other hand, using the weak scaling condition \eqref{a:wsc}, we have
\begin{align*}
\lv \int_{\ve r \ve < R} \frac{1-\cos \frac{r}{\zeta}}{\ve r \ve \varphi_k (\ve r \ve)} \, dr - \int_{\ve r \ve < R} \frac{r^2}{2\zeta^2 \ve r \ve \varphi_k(\ve r \ve)} \, dr \rv 
&\leq \frac{1}{24} \int_{\ve r \ve < R} \frac{r^4}{\zeta^4 \ve r \ve \varphi_k (\ve r \ve)} \, dr \\
&\leq \frac{a_k}{12 \zeta^4 (4-\os_k)} \frac{R^4}{\varphi_k(R)}.
\end{align*}
Since
\begin{align*}
0 \leq \lim_{k \rightarrow \infty} \frac{1}{\uC_{\varphi_k}(R)} \frac{a_k}{12\zeta^4 (4-\os_k)} \frac{R^4}{\varphi_k(R)} \leq \lim_{k \rightarrow \infty} \frac{a_k^2}{12\zeta^2} \frac{2-\us_k}{4-\us_k} R^2 = 0,
\end{align*}
we obtain
\begin{align*}
\lim_{k \rightarrow \infty} \frac{1}{\uC_{\varphi_k}(R)} \int_{\R} \frac{1-\cos \frac{r}{\zeta}}{\ve r \ve \varphi_k (\ve r \ve)} \, dr = \lim_{k \rightarrow \infty} \frac{1}{\uC_{\varphi_k}(R)} \frac{1}{\zeta^2} \int_0^R \frac{r}{\varphi_k (r)} \, dr = \frac{1}{\zeta^2}.
\end{align*}
Therefore, we conclude that
\begin{align*}
\lim_{k \rightarrow \infty} C_{\varphi_k} \uC_{\varphi_k}(R)
&= \lim_{k \rightarrow \infty} \left( \int_{\R^{n-1}} \int_{\R} \frac{1-\cos \frac{r}{\zeta}}{\ve r \ve \varphi_k (\ve r \ve)} \, dr \frac{dy'}{\zeta^n} \right)^{-1} \uC_{\varphi_k}(R) \\
&= \left( \int_{\R^{n-1}} \lim_{k \rightarrow \infty} \frac{1}{\uC_{\varphi_k}(R)} \int_{\R} \frac{1-\cos \frac{r}{\zeta}}{\ve r \ve \varphi_k(\ve r \ve)} \, dr \frac{dy'}{\zeta^n} \right)^{-1} \\
&= \left( \int_{\R^{n-1}} \frac{1}{\zeta^{n+2}} \, dy' \right)^{-1} = \frac{2}{\omega_n}.
\end{align*}
See \cite[Corollary 4.2]{DNPV} for the last equality.

We next assume that
\begin{align} \label{a:0}
\lim_{k \rightarrow \infty} \us_k = \lim_{k \rightarrow \infty} \os_k = 0.
\end{align}
We use the inequality \eqref{e:wsc 0} to compute
\begin{align*}
0 \leq \int_{\ve r \ve < R} \frac{1-\cos \frac{r}{\zeta}}{\ve r \ve \varphi_k (\ve r \ve)} \, dr \leq \frac{1}{\zeta^2} \int_0^{R} \frac{r}{\varphi_k(r)} \, dr \leq \frac{a_k}{\zeta^2 (2-\os_k)} \frac{R^2}{\varphi_k(R)}.
\end{align*}
Using the assumptions \eqref{a:a_k}, \eqref{a:0}, and the inequality \eqref{e:wsc infty}, we obtain
\begin{align} \label{e:<R}
\lim_{k\rightarrow \infty} \frac{1}{\oC_{\varphi_k}(R)} \int_{\ve r \ve < R} \frac{1-\cos \frac{r}{\zeta}}{\ve r \ve \varphi_k (\ve r \ve)} \, dr \leq \lim_{k \rightarrow \infty} \frac{a_k^2}{\zeta^2} \frac{\os_k}{2-\os_k} R^2 = 0.
\end{align}
On the other hand, observe that for any integer $m \geq 1$ we have 
\begin{align} \label{e:int A} 
\begin{split}
\lv \int_{2m \zeta \pi}^{2(m+1)\zeta \pi} \frac{\cos \frac{r}{\zeta}}{r \varphi_k(r)} \, dr \rv
&= \lv \int_{2m\zeta \pi}^{(2m+1)\zeta \pi} \left( \frac{\cos \frac{r}{\zeta}}{r \varphi_k (r)} + \frac{\cos(\frac{r}{\zeta}+\pi)}{(r+\zeta \pi) \varphi_k(r+ \zeta \pi)} \right) dr \rv \\
&\leq \int_{2m \zeta \pi}^{(2m +1) \zeta \pi} \lv \frac{1}{r \varphi_k(r)} - \frac{1}{(r+\zeta \pi) \varphi_k(r+\zeta \pi)} \rv dr \\
&= \int_{2m}^{2m+1} \lv \frac{1}{r\varphi_k(\zeta \pi r)} - \frac{1}{(r+1)\varphi_k(\zeta\pi(r+1))} \rv dr.
\end{split}
\end{align}
For the notational convenience, let us write $A = \frac{1}{r\varphi_k(\zeta \pi r)} - \frac{1}{(r+1) \varphi_k(\zeta\pi(r+1))}$. If $A \geq 0$, then by the weak scaling condition \eqref{a:wsc}, we have for $r \in [2m, 2m+1]$
\begin{align*}
\ve A \ve 
&\leq \frac{a_k}{\varphi_k((2m+1)\zeta\pi)} \frac{1}{r} \left( \frac{2m+1}{r} \right)^{\os_k} - \frac{1}{a_k \varphi_k((2m+1)\zeta\pi)} \frac{1}{r+1} \left( \frac{2m+1}{r+1} \right)^{\os_k} \\
&= \frac{(2m+1)^{\os_k}}{a_k \varphi_k((2m+1) \zeta\pi)} \frac{a_k^2 (r+1)^{1+\os_k} - r^{1+\os_k}}{r^{1+\os_k} (r+1)^{1+\os_k}} \\
&\leq \frac{1}{2a_k m \varphi((2m+1)\zeta \pi)} \left( \frac{2m+1}{2m} \right)^{\os_k} \frac{a_k^2 (r+1)^{1+\os_k} - r^{1+\os_k}}{(r+1)^{1+\os_k}} \\
&\leq \frac{9}{8a_k m \varphi((2m+1)\zeta \pi)} \frac{a_k^2 (r+1)^{1+\os_k} - r^{1+\os_k}}{(r+1)^{1+\os_k}}.
\end{align*}
If $A \leq 0$, then we have
\begin{align*}
\ve A \ve \leq \frac{9}{8a_k m \varphi_k((2m+1)\zeta \pi)} \frac{a_k^2 r^{1+\us_k} - (r+1)^{1+\us_k}}{(r+1)^{1+\us_k}}
\end{align*}
by similar argument. Since
\begin{align*}
\frac{a_k^2 (r+1)^{1+\os_k} - r^{1+\os_k}}{(r+1)^{1+\os_k}}
=& \int_0^1 \frac{d}{ds} \left( (1+(a_k^2-1)s) (r+s)^{1+\os_k} \right) \frac{1}{(r+1)^{1+\os_k}} \, ds \\
=& \int_0^1 \frac{(a_k^2 - 1)(r+s)^{1+\os_k}}{(r+1)^{1+\os_k}} \, ds \\
&+ \int_0^1 \frac{(1+\os_k)(1 + (a_k^2-1) s)(r+s)^{\os_k}}{(r+1)^{1+\os_k}} \, ds \\
\leq& a_k^2 - 1 + \frac{3a_k^2}{r+1}
\end{align*}
and
\begin{align*}
\frac{a_k^2 r^{1+\us_k} - (r+1)^{1+\us_k}}{(r+1)^{1+\us_k}}
=& \int_0^1 \frac{d}{ds} \left( (a_k^2s + 1-s)(r+1-s)^{1+\us_k} \right) \frac{1}{(r+1)^{1+\us_k}} \, ds \\
=& \int_0^1 \frac{(a_k^2-1)(r+1-s)^{1+\us_k}}{(r+1)^{1+\us_k}} \, ds \\
&- \int_0^1 \frac{(1+\us_k)(a_k^2s + 1 - s)(r+1-s)^{\us_k}}{(r+1)^{1+\us_k}} \, ds \leq a_k^2-1,
\end{align*}
we further estimate the integrand $\ve A \ve$ in \eqref{e:int A} as
\begin{align} \label{e:A}
\ve A \ve \leq \frac{9}{8a_k m \varphi((2m+1)\zeta\pi)} \left( a_k^2 - 1 + \frac{3a_k^2}{r+1} \right)
\end{align}
regardless of the sign of $A$.

Let $N \geq 1$ be the integer satisfying $2(N-1)\zeta\pi < R \leq 2N\zeta \pi$. Then, from \eqref{e:int A} and \eqref{e:A}, we have for $m \geq N$
\begin{align*}
\lv \int_{2m \zeta \pi}^{2(m+1)\zeta \pi} \frac{\cos \frac{r}{\zeta}}{r \varphi_k(r)} \, dr \rv
&\leq \frac{9}{8a_k m \varphi_k ((2m+1)\zeta \pi)} \int_{2m}^{2m+1} \left( a_k^2 - 1 +  \frac{3a_k^2}{r+1} \right) \, dr \\
&\leq \frac{9}{8a_k m \varphi_k ((2m+1)\zeta \pi)} \left( a_k^2 - 1 +  \frac{3a_k^2}{2m+1} \right) \\
&\leq \frac{9(a_k^2-1)}{8a_k} \frac{1}{m \varphi_k((2m+1) \zeta \pi)} + \frac{27a_k^2}{16 \varphi_k(R)} \frac{1}{m^2}.
\end{align*}
As a consequence, we have
\begin{align} \label{e:cos}
\begin{split}
\Bigg\ve \int_R^\infty &\frac{\cos \frac{r}{\zeta}}{r \varphi_k (r)} \, dr \Bigg\ve \leq \int_R^{2N\zeta \pi} \frac{dr}{r \varphi_k(r)} + \sum_{m = N}^\infty \lv \int_{2m\zeta \pi}^{2(m+1) \zeta \pi} \frac{\cos \frac{r}{\zeta}}{r \varphi_k (r)} \, dr \rv \\
&\leq \frac{a_k}{\varphi_k(R)} \int_R^{2N\zeta\pi} \frac{dr}{r} + \sum_{m=N}^\infty \left( \frac{9(a_k^2-1)}{8a_k m \varphi_k((2m+1) \zeta \pi)} + \frac{27a_k^2}{16 \varphi_k(R)} \frac{1}{m^2} \right) \\
&\leq \frac{a_k}{\varphi_k(R)} \log \frac{2N\zeta\pi}{R} + \frac{9(a_k^2 - 1)}{8a_k} \sum_{m=N}^\infty \frac{1}{m \varphi_k((2m+1) \zeta \pi)} + \frac{9\pi^2 a_k^2}{32 \varphi_k(R)}.
\end{split}
\end{align}
Now, we claim that 
\begin{align} \label{e:>R}
\lim_{k \rightarrow \infty} \frac{1}{\oC_{\varphi_k}(R)} \int_R^\infty \frac{\cos \frac{r}{\zeta}}{r \varphi_k(r)} \, dr = 0.
\end{align}
Indeed, by using the assumptions \eqref{a:a_k}, \eqref{a:0}, and the inequality \eqref{e:wsc infty}, the first and the third terms in \eqref{e:cos} can be handled as
\begin{align*}
0 \leq \lim_{k \rightarrow \infty} \frac{1}{\oC_{\varphi_k}(R)} \frac{a_k}{\varphi_k(R)} \log \frac{2N\zeta\pi}{R} \leq \lim_{k \rightarrow \infty} a_k^2 \os_k \log \frac{2N\zeta\pi}{R} = 0
\end{align*}
and
\begin{align*}
0 \leq \lim_{k \rightarrow \infty} \frac{1}{\oC_{\varphi_k}(R)} \frac{9 \pi^2 a_k^2}{32 \varphi_k(R)} \leq \lim_{k \rightarrow \infty} \frac{9 \pi^2 a_k^3}{32} \os_k = 0.
\end{align*}
For the second term, we first observe that
\begin{align*}
\oC_{\varphi_k}(R) 
&= \int_R^\infty \frac{1}{r \varphi_k(r)} \, dr \geq \sum_{m=N}^\infty \int_{2m\zeta \pi}^{2(m+1) \zeta \pi} \frac{1}{r \varphi_k(r)} \, dr \\
&\geq \frac{1}{a_k} \sum_{m=N}^\infty \int_{2m\zeta \pi}^{2(m+1) \zeta \pi} \frac{1}{2(m+1) \zeta \pi \varphi_k(2(m+1) \zeta \pi)} \, dr \\
&= \frac{1}{a_k} \sum_{m=N}^\infty \frac{1}{(m+1) \varphi_k(2(m+1) \zeta \pi)}.
\end{align*}
Since $m+1 \leq 2m$ and
\begin{align*}
\frac{\varphi_k(2(m+1) \zeta \pi)}{\varphi_k((2m+1) \zeta \pi)} \leq a_k \left( \frac{2m+2}{2m+1} \right)^{\os_k} \leq a_k \left( \frac{4}{3} \right)^2,
\end{align*}
we have
\begin{align*}
\oC_{\varphi_k}(R) \geq \frac{9}{16a_k^2} \sum_{m=N}^\infty \frac{1}{m \varphi_k((2m+1) \zeta \pi)},
\end{align*}
which yields that
\begin{align*}
\frac{1}{\oC_{\varphi_k}(R)} \sum_{m=N}^\infty \frac{1}{m \varphi_k((2m+1) \zeta \pi)} \leq \frac{16a_k^2}{9}.
\end{align*}
Thus, we obtain
\begin{align*}
0 \leq \lim_{k \rightarrow \infty} \frac{1}{\oC_{\varphi_k}(R)} \frac{9(a_k^2 - 1)}{8a_k} \sum_{m=N}^\infty \frac{1}{m \varphi_k((2m+1) \zeta \pi)} \leq \lim_{k \rightarrow \infty} 2a_k(a_k^2 -1) = 0,
\end{align*}
and this proves the claim. By \eqref{e:<R} and \eqref{e:>R}, we have
\begin{align*}
\lim_{k \rightarrow \infty} \frac{1}{\oC_{\varphi_k}(R)} \int_{\R} \frac{1-\cos \frac{r}{\zeta}}{\ve r \ve \varphi_k (\ve r \ve)} \, dr = \lim_{k \rightarrow \infty} \frac{1}{\oC_{\varphi_k}(R)} \int_{\ve r \ve \geq R} \frac{dr}{\ve r \ve \varphi_k (\ve r \ve)} = 2.
\end{align*}
Therefore, we conclude that
\begin{align*}
\lim_{k \rightarrow \infty} C_{\varphi_k} \oC_{\varphi_k}(R)
&= \lim_{k \rightarrow \infty} \left( \int_{\R^{n-1}} \int_{\R} \frac{1-\cos \frac{r}{\zeta}}{\ve r \ve \varphi_k(\ve r \ve)} \, dr \frac{dy'}{\zeta^n} \right)^{-1} \oC_{\varphi_k}(R) \\
&= \left( \int_{\R^{n-1}} \lim_{k \rightarrow \infty} \frac{1}{\oC_{\varphi_k}(R)} \int_{\R} \frac{1-\cos \frac{r}{\zeta}}{\ve r \ve \varphi_k(\ve r \ve)} \, dr \frac{dy'}{\zeta^n} \right)^{-1} \\
&= \left( \int_{\R^{n-1}} \frac{2}{\zeta^n} \, dy' \right)^{-1} = \frac{1}{n\omega_n},
\end{align*}
which finishes the proof. See \cite[Corollary 4.2]{DNPV} for the last equality.
\end{proof}

We next prove Proposition \ref{p:lim} using Lemma \ref{l:asymp}.

\begin{proof} [Proof of Proposition \ref{p:lim}]
Assume first that $\displaystyle \lim_{k \rightarrow \infty} \us_k = \lim_{k \rightarrow \infty} \os_k = 2$. In this case, we have no contribution outside the unit ball. Indeed, using inequality \eqref{e:wsc infty} we have
\begin{align*}
\lv -\int_{B_1^c} \frac{\delta(u, x, y)}{\ve y \ve^n \varphi_k(\ve y \ve)} \, dy \rv \leq 4n \omega_n \Ve u \Ve_{L^\infty(\Rn)} \int_1^\infty \frac{dr}{r\varphi_k (r)} \leq \frac{4n\omega_n a_k}{\us_k \varphi_k(1)} \Ve u \Ve_{L^\infty(\Rn)}.
\end{align*}
Hence, using the inequality \eqref{e:wsc 0} and the limit \eqref{e:asymp2} we obtain
\begin{align*}
\lv - \frac{1}{2} C_{\varphi_k} \int_{B_1^c} \frac{\delta(u,x,y)}{\ve y \ve^n \varphi_k (\ve y \ve)} \, dy \rv
&= \frac{1}{2\uC_{\varphi_k}(1)} C_{\varphi_k} \uC_{\varphi_k}(1) \lv \int_{B_1^c} \frac{\delta(u,x,y)}{\ve y \ve^n \varphi_k (\ve y \ve)} \, dy \rv \\
&\leq 2n \omega_n a_k^2 \frac{2-\us_k}{\us_k} C_{\varphi_k} \uC_{\varphi_k}(1) \Ve u \Ve_{L^\infty(\Rn)} \rightarrow 0
\end{align*}
as $k \rightarrow \infty$. On the other hand, we have
\begin{align*}
\lv \int_{B_1} \frac{\delta(u,x,y) - y \cdot D^2 u(x) y}{\ve y \ve^n \varphi_k (\ve y \ve)} \, dy \rv
&\leq \Ve u \Ve_{C^3(\Rn)} \int_{B_1} \frac{\ve y \ve^3}{\ve y \ve^n \varphi_k (\ve y \ve)} \, dy \\
&\leq n\omega_n \Ve u \Ve_{C^3(\Rn)} \int_0^1 \frac{r^2}{\varphi_k (r)} \, dr \\
&\leq \frac{n\omega_n a_k}{\varphi_k (1)(3-\os_k)} \Ve u \Ve_{C^3(\Rn)},
\end{align*}
and this implies that 
\begin{align*}
\lim_{k \rightarrow \infty} - \frac{1}{2} C_{\varphi_k} \int_{B_1} \frac{\delta(u,x,y)}{\ve y \ve^n \varphi_k (\ve y \ve)} \, dy = \lim_{k \rightarrow \infty} - \frac{1}{2} C_{\varphi_k} \int_{B_1} \frac{y \cdot D^2 u(x) y}{\ve y \ve^n \varphi_k (\ve y \ve)} \, dy.
\end{align*}
Note that if $i \neq j$ then
\begin{align*}
\int_{B_1} \frac{D_{ij} u(x) y_i y_j}{\ve y \ve^n \varphi_k (\ve y \ve)} \, dy = - \int_{B_1} \frac{D_{ij} u(x) \tilde{y}_i \tilde{y}_j}{\ve \tilde{y} \ve^n \varphi_k (\ve \tilde{y} \ve)} \, d\tilde{y},
\end{align*}
where $\tilde{y}_j = - y_j$ and $\tilde{y}_k = \tilde{y_k}$ for any $k \neq j$, and hence
\begin{align*}
\int_{B_1} \frac{D_{ij} u(x) y_i y_j}{\ve y \ve^n \varphi_k ( \ve y \ve)} \, dy = 0.
\end{align*}
Thus, we have
\begin{align*}
\int_{B_1} \frac{y \cdot D^2 u(x) y}{\ve y \ve^n \varphi_k (\ve y \ve)} \, dy 
&= \sum_{i=1}^n D_{ii} u(x) \int_{B_1} \frac{y_i^2}{\ve y \ve^n \varphi_k (\ve y \ve)} \, dy \\
&= \sum_{i=1}^n \frac{D_{ii} u(x)}{n} \int_{B_1} \frac{\ve y \ve^2}{\ve y \ve^n \varphi_k (\ve y \ve)} \, dy = \omega_n \Delta u(x) \int_0^1 \frac{r}{\varphi_k(r)} \, dr.
\end{align*}
Using \eqref{e:asymp2} we conclude that 
\begin{align*}
\lim_{k \rightarrow \infty } - L_k u(x) = \left( \lim_{k \rightarrow \infty} \frac{\omega_n}{2} C_{\varphi_k} \uC_{\varphi_k}(1) \right) (-\Delta u)(x) = - \Delta u(x).
\end{align*}

Next, we assume that $\displaystyle \lim_{k \rightarrow \infty} \us_k = \lim_{k \rightarrow \infty} \os_k = 0$. Fix $x \in \Rn$ and let $R_0 > 0$ be such that $\mathrm{supp} \, u \subset B_{R_0}$ and set $R = R_0 + \ve x \ve +1$. Then using the inequality \eqref{e:wsc 0} we have
\begin{align*}
\lv \int_{B_R} \frac{\delta(u,x,y)}{\ve y \ve^n \varphi_k (\ve y \ve)} \, dy \rv \leq n\omega_n \Ve u \Ve_{C^2(\Rn)} \int_0^R \frac{r}{\varphi_k (r)} \, dr \leq \frac{n\omega_n a_k}{2-\os_k} \frac{R^2}{\varphi_k (R)} \Ve u \Ve_{C^2(\Rn)}.
\end{align*}
Hence, using the inequality \eqref{e:wsc infty} and the limit \eqref{e:asymp0} we obtain
\begin{align*}
\lv -\frac{1}{2} C_{\varphi_k} \int_{B_R} \frac{\delta(u,x,y)}{\ve y \ve^n \varphi_k (\ve y \ve)} \, dy \rv
&\leq \frac{1}{2\oC_{\varphi_k}(R)} C_{\varphi_k} \oC_{\varphi_k}(R) \lv \int_{B_R} \frac{\delta(u,x,y)}{\ve y \ve^n \varphi_k (\ve y \ve)} \, dy \rv \\
&\leq \frac{n \omega_n a_k^2 R^2}{2} \frac{\os_k}{2-\os_k} C_{\varphi_k} \oC_{\varphi_k}(R) \Ve u \Ve_{C^2(\Rn)} \rightarrow 0
\end{align*}
as $k \rightarrow \infty$. On the other hand, if $\ve y \ve \geq R$, then $\ve x \pm y \ve > R_0$ and consequently $u(x \pm y) = 0$. Thus, we have
\begin{align*}
- \frac{1}{2} C_{\varphi_k} \int_{\Rn \setminus B_R} \frac{\delta(u,x,y)}{\ve y \ve^n \varphi_k (\ve y \ve)} \, dy = n \omega_n C_{\varphi_k} \oC_{\varphi_k}(R) u(x).
\end{align*}
Therefore, using \eqref{e:asymp0} we conclude that
\begin{align*}
\lim_{k \rightarrow \infty} L_k u(x) = n \omega_n \lim_{k \rightarrow \infty} C_{\varphi_k} \oC_{\varphi_k}(R) u(x) = u(x),
\end{align*}
which finishes the proof.
\end{proof}

The Lemma \ref{l:asymp} concerns about the limiting behavior of a sequence of constants $C_{\varphi_k}$, and does not provide an information about a fixed constant $C_\varphi$. To obtain uniform regularity estimates, we need uniform bounds for the constant $C_\varphi$ and these bounds will play an important role in the uniform estimates in the sequel.

\begin{lemma} \label{l:bound}
There exist constants $c_1, c_2 > 0$, depending only on $n$, such that for any $R > 0$
\begin{align} \label{e:bound}
\frac{c_1}{\uC(R) + \oC(R)} \leq C_\varphi \leq \frac{ac_2}{\uC(1) + \oC(1)}.
\end{align}
\end{lemma}

\begin{proof}
For the lower bound, notice that from the inequality $1-\cos \frac{r}{\zeta} \leq \frac{r^2}{2\zeta^2}$ we have
\begin{align*}
\int_{\R} \frac{1-\cos \frac{r}{\zeta}}{\ve r \ve \varphi(\ve r \ve)} \, dr = 2\int_0^R \frac{1-\cos \frac{r}{\zeta}}{r\varphi(r)} \, dr + 2\int_R^\infty \frac{1-\cos \frac{r}{\zeta}}{r \varphi(r)} \, dr \leq \frac{1}{\zeta^2} \uC(R) + 4\oC(R).
\end{align*}
Thus, it follows easily that
\begin{align*}
C_\varphi 
&= \left( \int_{\R^{n-1}} \int_{\R} \frac{1-\cos \frac{r}{\zeta}}{\ve r \ve \varphi(\ve r \ve)} \, dr \frac{dy'}{\zeta^n} \right)^{-1} \\
&\geq \left( \int_{\R^{n-1}} \left( \frac{1}{\zeta^2} \uC(R) + 4\oC(R) \right) \frac{dy'}{\zeta^n} \right)^{-1} \geq \frac{c_1(n)}{\uC(R) + \oC(R)}.
\end{align*}

For the upper bound, we first note that $1-\cos \frac{r}{\zeta} \geq \frac{r^2}{4\zeta^2}$ for $\ve r \ve \leq 1$ since $\zeta = (1+\ve y' \ve)^{1/2} \geq 1$. Thus we have
\begin{align} \label{e:0, infty}
\int_{\R} \frac{1-\cos \frac{r}{\zeta}}{\ve r \ve \varphi(\ve r \ve)} \, dr \geq \frac{1}{2\zeta^2} \uC + 2\int_1^\infty \frac{1-\cos \frac{r}{\zeta}}{r \varphi(r)} \, dr.
\end{align}
We next see that
\begin{align} \label{e:infty1}
\sum_{m=0}^\infty \int_{2m\zeta\pi + \frac{1}{2}\zeta\pi}^{2m\zeta\pi + \frac{3}{2}\zeta\pi} \frac{dr}{r\varphi(r)} \leq \int_1^\infty \frac{1-\cos \frac{r}{\zeta}}{r\varphi(r)} \, dr
\end{align}
and
\begin{align} \label{e:infty2}
\begin{split}
\sum_{m=0}^\infty \int_{2m\zeta\pi + \frac{3}{2}\zeta\pi}^{2m\zeta\pi + \frac{5}{2}\zeta\pi} \frac{dr}{r \varphi(r)}
&= \sum_{m=0}^\infty \int_{2m\zeta \pi + \frac{1}{2}\zeta\pi}^{2m\zeta\pi + \frac{3}{2}\zeta\pi} \frac{dr}{(r+\zeta\pi) \varphi(r+\zeta\pi)} \\
&\leq a \sum_{m=0}^\infty \int_{2m\zeta \pi + \frac{1}{2}\zeta\pi}^{2m\zeta\pi + \frac{3}{2}\zeta\pi} \frac{dr}{r \varphi(r)} \leq a\int_1^\infty \frac{1-\cos \frac{r}{\zeta}}{r\varphi(r)} \, dr.
\end{split}
\end{align}
If $r \geq 1+\zeta\pi$, then $r-\zeta \pi \geq \frac{1}{1+\zeta\pi} r$. Thus we have
\begin{align} \label{e:infty3}
\begin{split}
\int_1^{\frac{1}{2}\zeta\pi} \frac{dr}{r\varphi (r)} 
&= \int_{1+ \zeta \pi}^{\frac{3}{2} \zeta\pi} \frac{dr}{(r-\zeta\pi)\varphi (r-\zeta\pi)} \\
&\leq \int_{1+\zeta\pi}^{\frac{3}{2}\zeta\pi} \frac{a\left( 1+\zeta\pi \right)}{r \varphi(r)} \left( \frac{r}{r-\zeta\pi} \right)^{\os} \, dr \\
&\leq a \left( 1+\zeta\pi \right)^{1+\os} \int_{\frac{1}{2}\zeta\pi}^{\frac{3}{2}\zeta\pi} \frac{dr}{r \varphi(r)} \leq a \left( 1+\zeta\pi \right)^3 \int_1^\infty \frac{1-\cos \frac{r}{\zeta}}{r\varphi(r)} \, dr.
\end{split}
\end{align}
Combining \eqref{e:infty1}-\eqref{e:infty3}, we have
\begin{align*}
\int_1^\infty \frac{dr}{r\varphi(r)} \leq \left( 1+ a + a \left( 1+\zeta\pi \right)^3 \right) \int_1^\infty \frac{1-\cos\frac{r}{\zeta}}{r\varphi(r)} \, dr \leq Ca \zeta^3 \int_1^\infty \frac{1-\cos\frac{r}{\zeta}}{r\varphi(r)} \, dr.
\end{align*}
Putting this inequality into \eqref{e:0, infty}, we have
\begin{align*}
\int_{\R} \frac{1-\cos \frac{r}{\zeta}}{\ve r \ve \varphi(\ve r \ve)} \, dr \geq \frac{1}{2\zeta^2} \uC + \frac{C}{a\zeta^3} \oC.
\end{align*}
Therefore, we conclude that
\begin{align*}
C_\varphi \leq \left( \int_{\R^{n-1}} \left( \frac{1}{2\zeta^2} \uC + \frac{C}{a\zeta^3} \oC \right) \frac{dy'}{\zeta^n} \right)^{-1} \leq \frac{ac_2(n)}{\uC + \oC},
\end{align*}
which finishes the proof.
\end{proof}

\section{Viscosity Solutions} \label{s:3 VS}

In this section, we give a definition of viscosity solutions for integro-differential equations and a notion of the ellipticity as in \cite{CS}. We refer to \cite{CC} for the local equations. We begin with the notion of $C^{1,1}$ at the point.

\begin{definition}
A function $\psi$ is said to be $C^{1,1}$ at the point $x$, and we denote by $\psi \in C^{1,1}(x)$, if there is a vector $v \in \Rn$ and a number $M > 0$ such that
\begin{align*}
\ve \psi (x+y) - \psi(x) - v \cdot y \ve \leq M \ve y \ve^2 \quad \text{for} ~ \ve y \ve ~ \text{small enough}.
\end{align*}
We say that a function is $C^{1,1}$ in a set $\Omega$ if the previous definition holds at every point $x \in \Omega$ with a uniform constant $M$.
\end{definition}

We recall the definition of viscosity solutions for integro-differential equations.

\begin{definition} \label{d:vis}
A bounded function $u : \Rn \rightarrow \R$ which is upper (lower) semicontinuous in $\overline{\Omega}$ is said to be a {\it viscosity subsolution} ({\it viscosity supersolution}) to $Iu = f$, and we write $Iu \geq f$ ($Iu \leq f$), when the following holds: if a $C^2$-function $\psi$ touches $u$ from above (below) at $x \in \Omega$ in a small neighborhood $N$ of $x$, i.e., $\psi(x) = u(x)$ and $\psi > u$ in $N \setminus \lb x \rb$, then the function $v$ defined by 
\begin{align*}
v :=
\begin{cases}
\psi &\text{in} ~ N, \\
u &\text{in} ~ \Rn \setminus N,
\end{cases}
\end{align*}
satisfies $Iv(x) \geq f(x)$ ($Iv(x) \leq f(x)$). A function $u$ is said to be a {\it viscosity solution} if $u$ is both a viscosity subsolution and a viscosity supersolution.
\end{definition}

We can also give a definition of viscosity solutions to unbounded functions, but we will focus on bounded functions in this paper.

We next consider a collection $\mathcal{L}$ of linear integro-differential operators of the form \eqref{e:L symmetric} with kernels satisfying \eqref{a:K}. The {\it maximal operator} and the {\it minimal operator} with respect to $\mathcal{L}$ are defined as
\begin{align*}
\mathcal{M}^+_\mathcal{L} u = \sup_{L \in \mathcal{L}} Lu(x) \quad \text{and} \quad \mathcal{M}^-_\mathcal{L} u = \inf_{L \in \mathcal{L}} Lu(x).
\end{align*}
One example that we will use is the class $\mathcal{L}_0$. Recall that $\mathcal{L}_0$ is the class with kernels satisfying \eqref{a:K comp} additionally. In this case the maximal and the minimal operators are given by
\begin{align*}
&\mathcal{M}^+_{\mathcal{L}_0} u(x) = C_\varphi \int_{\Rn} \frac{\Lambda \delta(u,x,y)^+ - \lambda \delta(u,x,y)^-}{\ve y \ve^n \varphi(\ve y \ve)} \, dy \quad \text{and} \\
&\mathcal{M}^-_{\mathcal{L}_0} u(x) = C_\varphi \int_{\Rn} \frac{\lambda \delta(u,x,y)^+ - \Lambda \delta(u,x,y)^-}{\ve y \ve^n \varphi(\ve y \ve)} \, dy.
\end{align*}

Using these extremal operators, we give a general definition of ellipticity for nonlocal operators.

\begin{definition} \label{d:elliptic op}
Let $\mathcal{L}$ be a class of linear integro-differential operators. An {\it elliptic operator $I$ with respect to $\mathcal{L}$} is an operator with the following properties:
\begin{enumerate}
\item[(i)] If $u$ is bounded in $\Rn$ and is of $C^{1,1}(x)$, then $Iu(x)$ is defined classically.
\item[(ii)] If $u$ is bounded in $\Rn$ and is $C^2$ in some open set $\Omega$, then $Iu(x)$ is a continuous function in $\Omega$.
\item[(iii)] If $u$ and $v$ are bounded in $\Rn$ and are of $C^{1,1}(x)$, then
\begin{align*}
\mathcal{M}^-_\mathcal{L} (u-v) (x) \leq Iu(x) - Iv(x) \leq \mathcal{M}^-_\mathcal{L} (u-v) (x).
\end{align*}
\end{enumerate}
\end{definition}

For the nonlinear integro-differential operators of the form \eqref{e:I} we have the following properties: the proof can be found in \cite[Section 3 and 4]{CS}.

\begin{lemma} \label{l:classical}
Let $I$ be the operator of the form \eqref{e:I}. Then $I$ is an elliptic operator with respect to $\mathcal{L}_0$. Moreover, if $Iu \geq f$ in $\Omega$ in the viscosity sense and a function $\psi \in C^{1,1}(x)$ touches $u$ from above at $x$, then $Iu(x)$ is defined in the classical sense and $Iu(x) \geq f(x)$.
\end{lemma}

In \cite{CS} stability properties of viscosity solutions to the elliptic integro-differential equations with respect to the natural limits for lower-semicontinuous functions were proved. This type of limit is usually called a $\Gamma$-limit.

\begin{definition}
A sequence of lower-semicontinuous function $u_k$ $\Gamma${\it-converge} to $u$ in a set $\Omega$ if the followings hold:
\begin{enumerate}
\item[(i)] For every sequence $x_k \rightarrow x$ in $\Omega$,
\begin{align*}
\liminf_{k \rightarrow \infty} u_k(x_k) \geq u(x).
\end{align*}
\item[(ii)] For every $x \in \Omega$, there is a sequence $x_k \rightarrow x$ in $\Omega$ such that
\begin{align*}
\limsup_{k \rightarrow \infty} u_k(x_k) = u(x).
\end{align*}
\end{enumerate}
\end{definition}

Note that a uniformly convergent sequence $u_k$ also converges in the $\Gamma$ sense. We refer to \cite{CS} for the proof of the following lemma.

\begin{lemma}
Let $I$ be elliptic in the sense of Definition \ref{d:elliptic op} and $u_k$ be a sequence of functions that are uniformly bounded in $\Rn$ such that
\begin{enumerate}
\item[(i)] $Iu_k \leq f_k$ in $\Omega$ in the viscosity sense,
\item[(ii)] $u_k \rightarrow u$ in the $\Gamma$-sense in $\Omega$,
\item[(iii)] $u_k \rightarrow u$ a.e. in $\Rn$,
\item[(iv)] $f_k \rightarrow f$ locally uniformly in $\Omega$ for some continuous function $f$.
\end{enumerate}
Then $Iu \leq f$ in $\Omega$ in the viscosity sense.
\end{lemma}

\begin{corollary}
Let $I$ be elliptic in the sense of Definition \ref{d:elliptic op} and $u_k$ be a sequence of functions that are uniformly bounded in $\Rn$ such that
\begin{enumerate}
\item[(i)] $Iu_k = f_k$ in $\Omega$ in the viscosity sense,
\item[(ii)] $u_k \rightarrow u$ locally uniformly in $\Omega$,
\item[(iii)] $u_k \rightarrow u$ a.e. in $\Rn$,
\item[(iv)] $f_k \rightarrow f$ locally uniformly in $\Omega$ for some continuous function $f$.
\end{enumerate}
Then $Iu = f$ in $\Omega$ in the viscosity sense.
\end{corollary}

\begin{lemma}
Let $I$ be elliptic in the sense of Definition \ref{d:elliptic op}. Let $u$ and $v$ be bounded functions in $\Rn$ such that $Iu \geq f$ and $Iv \leq f$ in $\Omega$ in the viscosity sense for continuous functions $f$ and $g$. Then 
\begin{align*}
\mathcal{M}^+_\mathcal{L} (u - v) \geq f - g
\end{align*}
in $\Omega$ in the viscosity sense.
\end{lemma}

We check the Assumption 5.1 in \cite{CS} is true when $I$ is an elliptic operator with respect to $\mathcal{L}_0$ to prove the comparison principle. See also \cite[Lemma 5.10]{CS}.

\begin{theorem} [Comparison principle]
Let $I$ be an elliptic operator with respect to $\mathcal{L}_0$ in the sense of Definition \ref{d:elliptic op}. Let $\Omega \in \Rn$ be a bounded open set. If $u$ and $v$ are bounded functions in $\Rn$ such that $Iu \geq f$ and $Iv \leq f$ in $\Omega$ in the viscosity sense for some continuous function $f$, and $u \leq v$ in $\Rn \setminus \Omega$, then $u \leq v$ in $\Omega$.
\end{theorem}

\begin{proof}
The proof is the same as one for Theorem 5.2 in \cite{CS} if the Assumption 5.1 in \cite{CS} is provided. We claim that for every $R \geq 4$, there exists a constant $\delta = \delta(R) > 0$ such that $Lw_R > \delta$ in $B_R$ for any operator $L \in \mathcal{L}_0$, where $w_R(x) = \min \lbrace 1, \frac{\ve x \ve^2}{4R^2} \rbrace$. Indeed, for $x \in B_R$ we have
\begin{align*}
\delta(w_R, x, y) = \frac{\ve x+y \ve^2}{4R^2} + \frac{\ve x-y \ve^2}{4R^2} - \frac{2\ve x \ve^2}{4R^2} = \frac{\ve y \ve^2}{2R^2} \quad \text{if} ~ x \pm y \in B_{2R}
\end{align*}
and
\begin{align*}
\delta(w_R, x, y) \geq 1 - \frac{\ve x \ve^2}{2R^2} \geq 0 \quad \text{if} ~ x+y \not \in B_{2R} ~ \text{or} ~ x-y \not \in B_{2R}.
\end{align*}
Thus for any operator $L \in \mathcal{L}_0$ we obtain
\begin{align*}
Lw_R (x) \geq C_\varphi \lambda \int_{B_R} \frac{\ve y \ve^2}{2R^2} \frac{dy}{\ve y \ve^n \varphi (\ve y \ve)} := \delta(R) > 0,
\end{align*}
which proves the claim.
\end{proof}

\section{Regularity Results} \label{s:4 regularity}

In this section we prove Harnack inequality and H\"older regularity estimates for viscosity solutions of fully nonlinear elliptic integro-differential equations. From now on we will consider the class $\mathcal{L}_0$.

\subsection{Aleksandrov-Bakelman-Pucci(ABP) Estimates} \label{s:4.1 ABP}

We start this section with an ABP estimate which generalizes \cite[Theorem 8.7]{CS}. It is a fundamental tool in the proof of the Harnack inequality.

For a function $u$ that is not positive outside the ball $B_R$ we consider the concave envelope $\Gamma$ of $u^+$ in $B_{3R}$, which is defined by
\begin{align*}
\Gamma(x) :=
\begin{cases}
\min \lb p(x) : p ~ \text{is a plane such that} ~ p \geq u^+ ~ \text{in} ~ B_{3R} \rb &\text{in} ~ B_{3R}, \\
0 &\text{in} ~ \Rn \setminus B_{3R}.
\end{cases}
\end{align*}
We will focus on the contact set $\lb u = \Gamma \rb \cap B_R$ in the following lemmas.

\begin{lemma} \label{l:4.1}
Assume that $\mathcal{M}_{\mathcal{L}_0}^+ u \geq - f$ in $B_R$ in the viscosity sense and that $u \leq 0$ in $\Rn \setminus B_R$. Let $\Gamma$ be the concave envelope of $u^+$ in $B_{3R}$. Let $\rho_0 = 2^{-8} n^{-1}$ and $r_k = \rho_0 2^{-1/(2-\us)-k} R$. Then there exists a uniform constant $C > 0$, depending only on $n, \lambda$, and $a$, such that for each $x \in \lb u = \Gamma \rb$ and $M > 0$, we find $k \geq 0$ satisfying
\begin{align} \label{e:set}
\lv A_k \rv \leq C \frac{f(x)}{M} \ve B_{r_k}(x) \setminus B_{r_{k+1}}(x) \ve,
\end{align}
where
\begin{align*}
A_k = \lb y \in B_{r_k}(x) \setminus B_{r_{k+1}}(x) : u(y) < u(x) + (y-x) \cdot \nabla \Gamma(x) - M \frac{\uC + \oC}{\uC(r_0)} r_k^2 \rb.
\end{align*}
Here $\nabla \Gamma$ stands for an element of the superdifferential of $\Gamma$ at $x$.
\end{lemma}

\begin{proof}
Let $x$ be a point such that $u(x) = \Gamma(x) > 0$. By Lemma \ref{l:classical}, $\mathcal{M}^+_{\mathcal{L}_0} u(x)$ is defined classically and $\mathcal{M}^+_{\mathcal{L}_0} u(x) \geq -f(x)$. Note that if $x \pm y \in B_{3R}$ then $\delta(u,x,y) \leq 0$ since $\Gamma$ is concave and lies above $u$. Moreover, if either $x+y \not \in B_{3R}$ or $x-y \not \in B_{3R}$ then $x \pm y \not \in B_R$, which implies $u(x + y) \leq 0$ and $u(x-y) \leq 0$. In any case we have $\delta(u,x,y) \leq 0$ and hence
\begin{align} \label{e:f>0}
- f(x) \leq \mathcal{M}_{\mathcal{L}_0}^+ u(x) = C_\varphi \int_{\Rn} \frac{-\lambda \delta^-(u,x,y)}{\ve y \ve^n \varphi(\ve y \ve)} \, dy.
\end{align}
We split the integral as
\begin{align*}
f(x) \geq C_\varphi \lambda \sum_{k=0}^\infty \int_{A_k-x} \frac{\delta^-(u,x,y)}{\ve y \ve^n \varphi(\ve y \ve)} \, dy.
\end{align*}
If $y \in A_k - x$, then we have $x \pm y \in B_{3R}$ and
\begin{align*}
\delta(u,x,y)
&< \left( u(x) + y \cdot \nabla \Gamma(x) - M \frac{\uC + \oC}{\uC(r_0)} r_k^2 \right) + \left( \Gamma(x) - y \cdot \nabla \Gamma(x) \right) - 2u(x) \\
&= - M \frac{\uC + \oC}{\uC(r_0)} r_k^2,
\end{align*}
which yields
\begin{align*}
f(x) \geq C_\varphi \lambda a^{-1} \frac{\uC + \oC}{\uC(r_0)} \sum_{k=0}^\infty \frac{Mr_k^2}{r_k^n \varphi(r_k)} \lv A_k \rv
\end{align*}
with the help of the weak scaling condition \eqref{a:wsc}.

Suppose that we cannot find $k \in \mathbb{N} \cup \lb 0 \rb$ satisfying \eqref{e:set} with some constant $C > 0$ which will be chosen later. Then we have
\begin{align*}
f(x) 
&> C_\varphi \lambda a^{-1} \frac{\uC + \oC}{\uC(r_0)} \sum_{k=0}^\infty \frac{Mr_k^2}{r_k^n \varphi(r_k)} C \frac{f(x)}{M} \ve R_k \ve \\
&= \frac{3\omega_n \lambda}{4a} CC_\varphi \frac{\uC + \oC}{\uC(r_0)} \sum_{k=0}^\infty \frac{r_k^2}{\varphi(r_k)} f(x).
\end{align*}
We use the weak scaling condition to have
\begin{align*}
\uC(r_0)
&= \sum_{k=0}^\infty \int_{r_{k+1}}^{r_k} \frac{r}{\varphi(r)} \, dr \leq \sum_{k=0}^\infty \int_{r_{k+1}}^{r_k} \frac{ar_k}{\varphi(r_k)} \left( \frac{r_k}{r} \right)^{\os} \, dr \\
&\leq \sum_{k=0}^\infty \frac{a r_k (r_k - r_{k+1})}{\varphi(r_k)} \left( \frac{r_k}{r_{k+1}} \right)^2 = 4a \sum_{k=0}^\infty \frac{r_k^2}{\varphi(r_k)}.
\end{align*}
Using the inequality \eqref{e:bound} we arrive at
\begin{align*}
f(x) > \frac{3\omega_n \lambda}{16a^2} c_1 C f(x),
\end{align*}
which is a contradiction if we have taken $C \geq \frac{16a^2}{3\omega_n \lambda c_1}$.
\end{proof}

We observe from \eqref{e:f>0} that $f(x)$ is positive for $x \in \lb u = \Gamma \rb$.

\begin{lemma} \label{l:4.2}
Under the same assumptions as in Lemma \ref{l:4.1}, there exist uniform constants $\varepsilon_n \in (0,1)$ and $C = C(n, \lambda, a) > 0$ such that for each $x \in \lb u = \Gamma \rb$, we find some $k \geq 0$ satisfying
\begin{align*}
\frac{\lv \lb y \in B_{r_k}(x) \setminus B_{r_{k+1}}(x) : u(y) < u(x) + (y-x) \cdot \nabla \Gamma(x) - C \dfrac{\uC + \oC}{\uC(r_0)} f(x) r_k^2 \rb \rv}{\ve B_{r_k}(x) \setminus B_{r_{k+1}}(x) \ve} \leq \varepsilon_n
\end{align*}
and
\begin{align*}
\ve \nabla \Gamma (B_{r_k/4}(x)) \ve \leq C \left( \frac{\uC + \oC}{\uC(r_0)} f(x) \right)^n \ve B_{r_k/4}(x) \ve.
\end{align*}
\end{lemma}

For the proof of Lemma \ref{l:4.2} we refer to \cite[Lemma 8.4 and Corollary 8.5]{CS}. We next obtain a nonlocal ABP estimate.

\begin{theorem} [ABP estimate] \label{t:ABP}
Under the same assumptions as in Lemma \ref{l:4.1}, there is a finite, disjoint family of open cubes $\mathcal{Q}_j$ with diameters $d_j \leq r_0$ such that the followings hold:
\begin{enumerate}
\item[(i)] $\lb u = \Gamma \rb \cap \overline{\mathcal{Q}}_j \neq \emptyset$ for any $\mathcal{Q}_j$,
\item[(ii)] $\lb u = \Gamma \rb \subset \bigcup_{j=1}^m \overline{\mathcal{Q}}_j$,
\item[(iii)] $\ve \nabla \Gamma ( \overline{\mathcal{Q}}_j ) \ve \leq C \left(\dfrac{\uC + \oC}{\uC(r_0)} \max_{\overline{\mathcal{Q}}_j} f \right)^n \ve \mathcal{Q}_j \ve$,
\item[(iv)] $\lv \lb y \in 32 \sqrt{n} \mathcal{Q}_j : u(y) > \Gamma(y) - C \dfrac{\uC + \oC}{\uC(r_0)} \left( \max_{\overline{\mathcal{Q}}_j} f \right) d_j^2 \rb \rv \geq \mu_n \ve \mathcal{Q}_j \ve$
\end{enumerate}
for $\mu_n = 1-\varepsilon_n \in (0,1)$, where $C > 0$ is a uniform constant depending only on $n, \lambda$, and $a$.
\end{theorem}

The proof of Theorem \ref{t:ABP} can be found in \cite[Theorem 3.4]{CS}. It is important to note that when $\us$ is close to 2, the upper bound for the diameters $r_0 = \rho_0 2^{-1/(2-\us)}$ becomes very small so that Theorem \ref{t:ABP} generalizes the classical ABP estimate.

\subsection{A Barrier Function} \label{s:4.2 Barrier}

This section is devoted to construct a barrier function at every scale to find scaling invariant uniform estimates.

\begin{lemma} \label{l:subsoln}
Let $\kappa_1 \in (0,1), \sigma_0 \in (0,2)$, and assume $\us \geq \sigma_0$. There exist uniform constants $p = p(n, \lambda, \Lambda) > n+1$ and $\kappa_0 = \kappa_0(n, \lambda, \Lambda, a, \sigma_0) \in (0, \kappa_1/8)$ such that the function $\Phi_1(x) = \min \lb \ve \kappa_0 R \ve^{-p}, \ve x \ve^{-p} \rb$ satisfies
\begin{align*}
\mathcal{M}_{\mathcal{L}_0}^- \Phi_1(x) \geq 0
\end{align*}
for $x \in B_R \setminus B_{\kappa_1 R}$.
\end{lemma}

\begin{proof}
Without loss of generality we may assume that $x = R_0 e_1$ for $\kappa_1 R \leq R_0 < R$. We need to compute
\begin{align*}
\mathcal{M}_{\mathcal{L}_0}^- \Phi_1(x)
=& C_\varphi \int_{\Rn} \frac{ \lambda \delta^+(\Phi_1, x, y) - \Lambda \delta^- (\Phi_1, x, y) }{\ve y \ve^n \varphi( \ve y \ve )} \, dy \\
=& C_\varphi \int_{\Rn} \frac{\frac{\lambda}{2} \delta^+}{\ve y \ve^n \varphi( \ve y \ve )} \, dy + C_\varphi \int_{B_{R_0/2}} \frac{\frac{\lambda}{2} \delta^+ - \Lambda \delta^-}{\ve y \ve^n \varphi(\ve y \ve)} \, dy \\
&+ C_\varphi \int_{\Rn \setminus B_{R_0 / 2}} \frac{\frac{\lambda}{2} \delta^+ - \Lambda \delta^-}{\ve y \ve^n \varphi(\ve y \ve)} \, dy =: C_\varphi \left( I_1 + I_2 + I_3 \right).
\end{align*}
For $\ve y \ve \leq R_0 / 2$, we have
\begin{align*}
\delta(\Phi_1, x, y)
&= R_0^{-p} \left( \lv \frac{x}{R_0} + \frac{y}{R_0} \rv^{-p} + \lv \frac{x}{R_0} - \frac{y}{R_0} \rv^{-p} - 2 \right) \\
&\geq pR_0^{-p} \left( - \lv \frac{y}{R_0} \rv^2 + (p+2) \frac{y_1^2}{R_0^2} - \frac{1}{2} (p+2)(p+4) \frac{y_1^2 \ve y \ve^2}{R_0^4} \right).
\end{align*}
We choose $p = p(n, \lambda, \Lambda) > n + 1$ large enough so that
\begin{align*}
(p+2) \frac{\lambda}{2} \int_{\partial B_1} y_1^2 \, d\sigma(y) - \Lambda \ve \partial B_1 \ve \geq 0.
\end{align*}
Then we obtain
\begin{align*}
I_2
&\geq p R_0^{-p} \int_{B_{R_0/2}} \left( \frac{\lambda}{2} (p+2) \frac{y_1^2}{R_0^2} - \Lambda \left( \frac{\ve y \ve^2}{R_0^2} + \frac{(p+2)(p+4) y_1^2 \ve y \ve^2}{2R_0^4} \right) \right) \frac{1}{\ve y \ve^n \varphi(\ve y \ve)} \, dy \\
&\geq -p R_0^{-p} \frac{\Lambda(p+2)(p+4)}{2R_0^4} c_n \int_0^{R_0/2} \frac{r^3}{\varphi(r)} \, dr,
\end{align*}
where $c_n = \int_{\partial B_1} y_1^2 \, d\sigma(y) > 0$ is a constant depending only on $n$. Using the weak scaling condition \eqref{a:wsc}, we have
\begin{align*}
\int_0^{R_0/2} \frac{r^3}{\varphi(r)} \, dr \leq \frac{a}{4-\os} \frac{R_0^4}{16\varphi(R_0/2)} \leq \frac{aR_0^4}{32 \varphi(R_0/2)}
\end{align*}
and hence 
\begin{align} \label{e:I_2}
I_2 \geq -\frac{\Lambda p(p+2)(p+4)c_n a}{64R_0^p \varphi(R_0/2)}.
\end{align}
On the other hand, using the inequality \eqref{e:wsc infty} we estimate
\begin{align} \label{e:I_3}
I_3 \geq - \int_{B_{R_0/2}^c} \frac{2\Lambda R_0^{-p}}{\ve y \ve^n \varphi(\ve y \ve)} \, dy = - 2n \omega_n \Lambda R_0^{-p} \oC(R_0/2) \geq - \frac{2n \omega_n \Lambda}{a \us R_0^p \varphi(R_0/2)}.
\end{align}
Now we will make $I_1$ sufficiently large by selecting $\kappa_0 > 0$ small. We have
\begin{align*}
I_1
&\geq \frac{\lambda}{2} \int_{B_{R_0/4}(x)} \frac{\ve x-y \ve^{-p} - 2R_0^{-p}}{\ve y \ve^n \varphi( \ve y \ve)} \, dy \geq \frac{\lambda}{4} \int_{B_{R_0/4}(x) \setminus B_{\kappa_0 R} (x) } \frac{\ve x-y \ve^{-p}}{\ve y \ve^n \varphi(\ve y \ve)} \, dy \\
&= \frac{\lambda}{4} \int_{B_{R_0/4}(0) \setminus B_{\kappa_0 R} (0) } \frac{\ve z \ve^{-p}}{\ve x+z \ve^n \varphi(\ve x+z \ve)} \, dz \\
&\geq \frac{\lambda n \omega_n}{2^{n+2} R_0^n} \left( \min_{r \in \left[ \frac{R_0}{2}, \frac{3R_0}{2} \right]} \frac{1}{\varphi(r)} \right) \int_{\kappa_0 R}^{R_0/4} r^{-p+n-1} \, dr.
\end{align*}
If we have taken $\kappa_0 \in (0, \kappa_1/8)$, then we have
\begin{align*}
\int_{\kappa_0 R}^{R_0/4} r^{-p+n-1} \, dr 
&= \frac{(\kappa_0 R)^{-p+n} - (R_0 / 4)^{-p+n}}{p-n} \\
&\geq \frac{(\kappa_0 / \kappa_1)^{p-n} - 4^{p-n}}{p-n} R_0^{n-p} \geq \frac{1}{2(p-n)} \frac{\kappa_1}{\kappa_0} R_0^{n-p}.
\end{align*}
We use the weak scaling condition \eqref{a:wsc} to obtain
\begin{align*}
\min_{r \in \left[ \frac{R_0}{2}, \frac{3R_0}{2} \right]} \frac{1}{\varphi(r)} \geq \frac{1}{a \varphi(3R_0/2)} \geq \frac{1}{a^2 3^{\os} \varphi(R_0/2)} \geq \frac{1}{9a^2 \varphi(R_0/2)},
\end{align*}
and hence
\begin{align} \label{e:I_1}
I_1 \geq \frac{\lambda n \omega_n}{9\cdot 2^{n+3} a^2 (p-n)} \frac{\kappa_1}{\kappa_0} \frac{1}{R_0^p \varphi(R_0/2)}.
\end{align}
Combining \eqref{e:I_2}-\eqref{e:I_1}, we have
\begin{align*}
I_1 + I_2 + I_3 \geq \left( \frac{\lambda n \omega_n}{9 \cdot 2^{n+3} a^2 (p-n)} \frac{\kappa_1}{\kappa_0} - \frac{\Lambda p(p+2)(p+4)c_n a}{64} - \frac{2\Lambda n \omega_n}{a \sigma_0} \right) \frac{R_0^{-p}}{\varphi(R_0/2)}.
\end{align*}
By taking $\kappa_0 = \kappa_0(n, \lambda, \Lambda, a, \sigma_0) \in (0, \kappa_1/8)$ small enough, we have $\mathcal{M}_{\mathcal{L}_0}^- \Phi_1(x) \geq 0$ for $x \in B_R \setminus B_{\kappa_1 R}$.
\end{proof}

\subsection{Power Decay Estimates} \label{s:4.3 decay est}

In this section we establish the measure estimates of super-level sets of the viscosity supersolutions to fully nonlinear elliptic integro-differential equations with respect to $\mathcal{L}_0$ using the ABP estimates and the barrier function constructed in Lemma \ref{l:subsoln}. Let $Q_R = Q_R(0)$ denote a dyadic cube of side $R$ centered at 0 in the sequel.

\begin{lemma} \label{l:decay}
Assume $\us \geq \sigma_0 > 0$. There exist uniform constants $\varepsilon_0, \mu_0 \in (0,1)$ and $M_0 > 1$, depending on $n, \lambda, \Lambda, a$ and $\sigma_0$, such that if a nonnegative function $u$ satisfies $\inf_{Q_\frac{3R}{2\sqrt{n}}} u \leq 1$ and $\mathcal{M}^-_{\mathcal{L}_0} u \leq \frac{\uC(R)}{R^2(\uC + \oC)} \varepsilon_0$ in $Q_{2R}$ in the viscosity sense, then 
\begin{align*}
\lv \lb u \leq M_0 \rb \cap Q_\frac{R}{2\sqrt{n}} \rv > \mu_0 \lv Q_\frac{R}{2\sqrt{n}} \rv.
\end{align*}
\end{lemma}

\begin{proof}
Let $\Phi_1$ be the function in Lemma \ref{l:subsoln} with $\kappa_1 = \rho_0$. Define
\begin{align*}
\Phi(x) := c_0
\begin{cases}
P(x) &\text{for } x \in B_{\kappa_0 R}, \\
(\kappa_0 R)^p ( \Phi_1(x) - \Phi_1(R) ) = (\kappa_0 R)^p \left( \ve x \ve^{-p} - R^{-p} \right) &\text{for } x \in B_R \setminus B_{\kappa_0 R}, \\
0 &\text{for } x \in B_R^c,
\end{cases}
\end{align*}
where $c_0 = \frac{2}{\kappa_0^p ((4/3)^p - 1)}$ and $P(x) := -a\ve x \ve^2 + b$ with $a = \frac{1}{2} p(\kappa_0 R)^{-2}$ and $b = 1- \kappa_0^p + \frac{1}{2}p$. Then $\Phi$ is a $C^{1,1}$ function on $B_R$ and $\Phi \geq 2$ in $B_{3R/4}$. If $x \in B_R \setminus B_{\kappa_0 R}$, then
\begin{align*}
\delta(\Phi, x, y) 
&= \Phi(x+y) + \Phi(x-y) - 2\Phi(x) \\
&\geq c_0 (\kappa_0 R)^p \left( \Phi_1(x+y) - R^{-p} + \Phi_1(x-y) - R^{-p} - 2\Phi_1(x) + 2R^{-p} \right) \\
&= c_0(\kappa_0 R)^p \delta(\Phi_1, x, y).
\end{align*}
Thus, we have
\begin{align*}
\mathcal{M}^-_{\mathcal{L}_0} \Phi(x) 
&= C_\varphi \int_{\Rn} \frac{\lambda \delta^+(\Phi, x, y) - \Lambda \delta^-(\Phi, x, y)}{\ve y \ve^n \varphi(\ve y \ve)} \, dy \\
&\geq C_\varphi c_0 (\kappa_0 R)^p \int_{\Rn} \frac{\lambda \delta^+(\Phi_1, x, y) - \Lambda \delta^-(\Phi_1, x, y)}{\ve y \ve^n \varphi(\ve y \ve)} \, dy \geq 0.
\end{align*}
If $x \in B_R^c$, then we have $\delta(\Phi, x, y) \geq 0$ and hence $\mathcal{M}^-_{\mathcal{L}_0} \Phi (x) \geq 0$. Finally, if $x \in B_{\kappa_0 R}$, then we have
\begin{align*}
\mathcal{M}_{\mathcal{L}_0}^-\Phi(x)
&\geq - C_\varphi \Lambda \int_{\Rn} \frac{\delta^-(\Phi, x, y)}{\ve y \ve^n \varphi(\ve y \ve)} \, dy \\
&\geq - C_\varphi \Lambda \int_{B_R} \frac{c R^{-2} \ve y \ve^2}{\ve y \ve^n \varphi(\ve y \ve)} \, dy - 2bc_0 C_\varphi \Lambda \int_{B_R^c} \frac{1}{\ve y \ve^n \varphi(\ve y \ve)} \, dy \\
&= - C_\varphi c \Lambda n \omega_n R^{-2} \uC(R) - 2bc_0 C_\varphi \Lambda n \omega_n \oC(R)
\end{align*}
since $D^2 \Phi \geq - cR^{-2} I$ a.e. in $B_R$ for some constant $c>0$. This implies that
\begin{align*}
\mathcal{M}_{\mathcal{L}_0}^- \Phi \geq -\psi \quad \text{in} ~ \Rn
\end{align*}
for some function with $\mathrm{supp} \, \psi \in B_{\rho_0 R}$ and a uniform bound
\begin{align} \label{e:psi}
\psi \leq C_\varphi c \Lambda n \omega_n R^{-2} \uC(R) + 2bc_0 C_\varphi \Lambda n \omega_n \oC(R).
\end{align}

We now consider the function $v := \Phi - u$. It satisfies that $v \leq 0$ outside $B_R$, $\max_{B_R} v \geq 1$, and
\begin{align*}
\mathcal{M}_{\mathcal{L}_0}^+ v \geq \mathcal{M}_{\mathcal{L}_0}^- \Phi - \mathcal{M}_{\mathcal{L}_0}^- u \geq - \psi - \frac{\uC(R)}{R^2(\uC + \oC)} \varepsilon_0
\end{align*}
in $B_R$. For the concave envelope $\Gamma$ of $u^+$ in $B_{3R}$, by Theorem \ref{t:ABP}, we have
\begin{align*}
\frac{1}{R} 
&\leq \frac{1}{R} \max_{B_R} v \leq C \ve \nabla \Gamma (B_R) \ve^{1/n} \leq C \left( \sum_j \ve \nabla \Gamma(\overline{\mathcal{Q}_j}) \ve \right)^{1/n} \\
&\leq C \left( \sum_j C \left( \frac{\uC + \oC}{\uC(r_0)} \left( \psi + \frac{\uC(R)}{R^2(\uC + \oC)} \varepsilon_0 \right) \right)^n \ve \mathcal{Q}_j \ve \right)^{1/n} \\
&\leq \frac{C}{R^2} \left( \sum_j \left( \frac{\uC + \oC}{\uC(r_0)} R^2 \psi + \frac{\uC(R)}{\uC(r_0)} \varepsilon_0 \right)^n \ve \mathcal{Q}_j \ve \right)^{1/n}.
\end{align*}
Since $\mathrm{supp} \, \psi \in B_{\rho_0 R}$ and $\sum_j \ve \mathcal{Q}_j \ve \leq C \ve B_R \ve$, it follows that
\begin{align*}
\frac{1}{R} 
&\leq \frac{C}{R^2} \left( \sum_{\overline{\mathcal{Q}_j} \cap B_{\rho_0 R} \neq \emptyset} \left( \frac{\uC + \oC}{\uC(r_0)} R^2 \psi \right)^n \ve \mathcal{Q}_j \ve \right)^{1/n} + \frac{C}{R} \frac{\uC(R)}{\uC(r_0)} \varepsilon_0.
\end{align*}
Using \eqref{e:psi} and \eqref{e:bound} we have
\begin{align*}
\frac{1}{R} \leq \frac{C}{R^2} \left( \sum_{\overline{\mathcal{Q}_j} \cap B_{\rho_0 R} \neq \emptyset} \left( \frac{\uC(R)}{\uC(r_0)} + \frac{R^2 \oC(R)}{\uC(r_0)} \right)^n \ve \mathcal{Q}_j \ve \right)^{1/n} + \frac{C}{R} \frac{\uC(R)}{\uC(r_0)} \varepsilon_0.
\end{align*}
We use the inequality \eqref{e:wsc ratio} to obtain
\begin{align*}
\frac{\uC(R)}{\uC(r_0)} \leq 1+ a^2 \left( \frac{R}{r_0} \right)^{2-\us} = 1 + \frac{2a^2}{\rho_0^{2-\us}} \leq 1+ \frac{2a^2}{\rho_0^2},
\end{align*}
and use the inequalities \eqref{e:wsc 0}, \eqref{e:wsc infty} to obtain
\begin{align*}
\frac{R^2 \oC(R)}{\uC(r_0)} 
&\leq R^2 \left( \frac{a}{\varphi(R) \us} \right) \left( \frac{r_0^2}{a \varphi(r_0) (2-\us)} \right)^{-1} \\
&= a^2 \frac{2-\us}{\us} \frac{R^2}{r_0^2} \frac{\varphi(r_0)}{\varphi(R)} \leq \frac{2a^3}{\sigma_0} \left( \frac{R}{r_0} \right)^{2-\us} = \frac{4a^3}{\rho_0^2 \sigma_0}.
\end{align*}
Therefore, it follows that
\begin{align*}
1 \leq \frac{C}{R} \left( \sum_{\overline{\mathcal{Q}_j} \cap B_{\rho_0 R} \neq \emptyset} \ve \mathcal{Q}_j \ve \right)^{1/n} + C\varepsilon_0.
\end{align*}
By taking $\varepsilon_0 > 0$ small, we have
\begin{align*}
\lv Q_\frac{R}{2\sqrt{n}} \rv \leq C \sum_{\overline{\mathcal{Q}_j} \cap B_{\rho_0 R} \neq \emptyset} \ve \mathcal{Q}_j \ve,
\end{align*}
for some constant $C>0$ depending on $n, \lambda, \Lambda, a$, and $\sigma_0$. We now use Theorem \ref{t:ABP} to obtain
\begin{align*}
\mu \ve \mathcal{Q}_j \ve 
&\leq \lv \lb y \in 32\sqrt{n}\mathcal{Q}_j : v(y) > \Gamma(y) - C \frac{\uC + \oC}{\uC(r_0)} C_\varphi \left( R^{-2} \uC(R) + \oC(R) \right) d_j^2 \rb \rv \\
&\leq \lv \lb y \in 32\sqrt{n} \mathcal{Q}_j : u(y) \leq \Phi(y) + C \rb \rv \leq \lv \lb y \in 32\sqrt{n} \mathcal{Q}_j : u(y) \leq M_0 \rb \rv
\end{align*}
for some constant $M_0 = \Ve \Phi \Ve_\infty + C$, depending only on $n, \lambda, \Lambda, a$, and $\sigma_0$, where we have used that $d_j \leq R$. Since $\rho_0 = 2^{-8} n^{-1}$, we know that $32\sqrt{n} \mathcal{Q}_j \subset B_\frac{R}{4\sqrt{n}}$ for any $\mathcal{Q}_j$ satisfying $\overline{\mathcal{Q}}_j \cap B_{\rho_0 R} \neq \emptyset$. Taking a subcover of $\lb 32\sqrt{n} \mathcal{Q}_j : \overline{\mathcal{Q}}_j \cap B_{\rho_0 R} \neq \emptyset \rb$ with finite overlapping, we obtain
\begin{align*}
\lv Q_{\frac{R}{2\sqrt{n}}} \rv \leq C \sum_{\overline{\mathcal{Q}_j} \cap B_{\rho_0 R} \neq \emptyset} \ve \mathcal{Q}_j \ve \leq C \lv \lb u \leq M_0 \rb \cap B_{\frac{R}{4\sqrt{n}}} \rv \leq C \lv \lb u \leq M_0 \rb \cap Q_{\frac{R}{2\sqrt{n}}} \rv.
\end{align*}
Thus, we have $\lv \lb u \leq M_0 \rb \cap Q_{\frac{R}{2\sqrt{n}}} \rv > \mu_0 \lv Q_{\frac{R}{2\sqrt{n}}} \rv$.
\end{proof}

\begin{corollary}
Under the same assumptions as in Lemma \ref{l:decay}, we have
\begin{align*}
\lv \lb u > M_0^k \rb \cap Q_\frac{R}{2\sqrt{n}} \rv \leq (1-\mu_0)^k \lv Q_\frac{R}{2\sqrt{n}} \rv
\end{align*}
for all $k \in \mathbb{N}$, and hence
\begin{align*}
\lv \lb u > t \rb \cap Q_\frac{R}{2\sqrt{n}} \rv \leq C R^n t^{-\varepsilon}
\end{align*}
for all $t > 0$, where $C$ and $\varepsilon$ are uniform constants.
\end{corollary}

By the standard covering argument, we deduce the weak Harnack inequality as follows.

\begin{theorem} [Weak Harnack inequality] \label{t:weakHI}
Assume $\us \geq \sigma_0 > 0$. Let $u$ be a nonnegative function in $\mathbb{R}^n$ such that
\begin{align*}
\mathcal{M}_{\mathcal{L}_0}^- u \leq C_0 \quad \text{in} ~ B_{2R}
\end{align*}
in the viscosity sense. Then we have
\begin{align*}
\lv \lb u > t \rb \cap B_R \rv \leq CR^n \left( u(0) + C_0 R^2 \frac{\uC + \oC}{\uC(R)} \right)^\varepsilon t^{-\varepsilon} \quad \text{for all} ~ t > 0,
\end{align*}
and hence
\begin{align*}
\left( \fint_{B_R} \ve u \ve^p \right)^{1/p} \leq C \left( u(0) + C_0 R^2 \frac{\uC + \oC}{\uC(R)} \right),
\end{align*}
where $C > 0, \varepsilon > 0$, and $p > 0$ are uniform constants depending only on $n, \lambda, \Lambda, a$, and $\sigma_0$.
\end{theorem}

\subsection{Harnack Inequality} \label{s:4.4 Harnack}

This section is devoted to the proof of Harnack inequality for fully nonlinear elliptic integro-differential operators with respect to $\mathcal{L}_0$, where the constant depends only on $n, \lambda ,\Lambda, a$, and $\sigma_0$.

\begin{theorem} [Harnack inequality]
Assume $\us \geq \sigma_0 > 0$. Let $u \in C(B_{2R})$ be a nonnegative function in $\Rn$ such that
\begin{align*}
\mathcal{M}^-_{\mathcal{L}_0} u \leq C_0 \frac{\uC(R)}{(\uC + \oC)R^2} \quad \text{and} \quad \mathcal{M}^+_{\mathcal{L}_0} u \geq - C_0 \frac{\uC(R)}{(\uC + \oC)R^2} \quad \text{in} ~ B_{2R}
\end{align*}
in the viscosity sense. Then there exists a uniform constant $C > 0$, depending only on $n, \lambda, \Lambda, a$, and $\sigma_0$, such that
\begin{align*}
\sup_{B_{R/2}} u \leq C \left( u(0) + C_0 \right).
\end{align*}
\end{theorem}

\begin{proof}
We may assume that $u > 0$, $u(0) \leq 1$, and $C_0 = 1$. Let $\varepsilon > 0$ be the constant as in Theorem \ref{t:weakHI} and let $\gamma = (n+2)/\varepsilon$. Consider the minimal value of $\alpha > 0$ such that 
\begin{align*}
u(x) \leq h_\alpha(x) := \alpha \left( 1- \frac{\ve x \ve}{R} \right)^{-\gamma} \quad \text{for all} ~ x \in B_R,
\end{align*}
so that there exists $x_0 \in B_R$ satisfying $u(x_0) = h_\alpha(x_0)$. It is enough to show that $\alpha$ is uniformly bounded. 

Let $d = R - \ve x_0 \ve, r = d/2$, and let $A = \lb u > u(x_0) / 2 \rb$. Then $u(x_0) = h_\alpha(x_0) = \alpha R^\gamma d^{-\gamma}$. By the weak Harnack inequality, we have
\begin{align*}
\lv A \cap B_R \rv \leq C R^n \left( \frac{4}{u(x_0)} \right)^\varepsilon \leq C\alpha^{-\varepsilon} R^n \left( \frac{d}{R} \right)^{\gamma \varepsilon} \leq C \alpha^{-\varepsilon} d^n.
\end{align*}
Since $B_r(x_0) \subset \subset B_R$ and $r = d/2$, we obtain
\begin{align*}
\lv \lb u > u(x_0) / 2 \rb \cap B_r(x_0) \rv \leq C \alpha^{-\varepsilon} \ve B_r(x_0) \ve.
\end{align*}
We will show that there exists a uniform constant $\theta > 0$ such that $\ve \lb u < u(x_0) / 2 \rb \cap B_{\theta r/4}(x_0) \ve \leq \frac{1}{2} \ve B_{\theta r/4} \ve$ for a large constant $\alpha > 1$, which yields that $\alpha > 0$ is uniformly bounded.

We first estimate $\lv \lb u < u(x_0) / 2 \rb \cap B_{\theta r}(x_0) \rv$ for small $\theta > 0$, which will be chosen uniformly later. For every $x \in B_{\theta r}(x_0)$,
\begin{align*}
u(x) \leq h_\alpha(x) \leq \alpha \left( \frac{d-\theta r}{R} \right)^{-\gamma} = \alpha \left( \frac{d}{R} \right)^{-\gamma} \left( 1- \frac{\theta}{2} \right)^{-\gamma} = \left( 1- \frac{\theta}{2} \right)^{-\gamma} u(x_0). 
\end{align*}
Consider the function
\begin{align*}
v(x) := \left( 1- \frac{\theta}{2} \right)^{-\gamma} u(x_0) - u(x).
\end{align*}
Note that $v$ is nonnegative in $B_{\theta r}(x_0)$. To apply the weak Harnack inequality to $w := v^+$, we compute $\mathcal{M}_{\mathcal{L}_0}^- w$ in $B_{\theta r}(x_0)$. For $x \in B_{\theta r}(x_0)$, 
\begin{align}
\mathcal{M}^-_{\mathcal{L}_0} w(x)
&\leq \mathcal{M}_{\mathcal{L}_0}^- v(x) + \mathcal{M}_{\mathcal{L}_0}^+ v^-(x) \nonumber \\
&\leq - \mathcal{M}^+_{\mathcal{L}_0} u(x) + C_\varphi \int_{\Rn} \frac{\Lambda v^-(x+y) + \Lambda v^-(x-y)}{\ve y \ve^n \varphi(\ve y \ve)} \, dy \nonumber \\
&= \frac{\uC(R)}{R^2(\uC + \oC)} + 2\Lambda C_\varphi \int_{\lb v(x+y) < 0 \rb} \frac{v^-(x+y)}{\ve y \ve^n \varphi(\ve y \ve)} \, dy \nonumber \\
&\leq \frac{\uC(R)}{R^2(\uC + \oC)} + 2\Lambda C_\varphi \int_{B^c_{\theta r}(x_0 - x)} \frac{\left( u(x+y) - (1-\frac{\theta}{2})^{-\gamma} u(x_0) \right)^+}{\ve y \ve^n \varphi(\ve y \ve)} dy \label{e:Mw}
\end{align}
in the viscosity sense.

Consider the largest number $\beta > 0$ such that
\begin{align*}
u(x) \geq g_\beta(x) := \beta \left( 1- \frac{\ve 4x \ve^2}{R^2} \right)^+,
\end{align*}
and let $x_1 \in B_\frac{R}{4}$ be a point such that $u(x_1) = g_\beta(x_1)$. This is possible because we have assumed that $u > 0$ in $B_{2R}$. We observe that $\beta \leq 1$ since $u(0) \leq 1$. We estimate
\begin{align*}
C_\varphi \int_{\Rn} \frac{\delta^-(u, x_1, y)}{\ve y \ve^n \varphi(\ve y \ve)} \, dy
&\leq C_\varphi \int_{\Rn} \frac{\delta^-(g_\beta, x_1, y)}{\ve y \ve^n \varphi( \ve y \ve)} \, dy \\
&\leq CC_\varphi \left( \int_{B_R} \frac{\ve y \ve^2}{R^2} \frac{1}{\ve y \ve^n \varphi( \ve y \ve)} \, dy + \int_{\Rn \setminus B_R} \frac{dy}{\ve y \ve^n \varphi(\ve y \ve)} \right) \\
&\leq \frac{C}{\uC + \oC} \left( \frac{\uC(R)}{R^2} + \oC(R) \right).
\end{align*}
Since
\begin{align*}
\frac{\uC (R)}{R^2(\uC + \oC)} \geq \mathcal{M}^-_{\mathcal{L}_0} u = C_\varphi \int_{\Rn} \frac{\lambda \delta^+(u, x_1, y) - \Lambda \delta^-(u, x_1, y)}{\ve y \ve^n \varphi(\ve y \ve)} \, dy, 
\end{align*}
we obtain
\begin{align*}
C_\varphi \int_{\Rn} \frac{\delta^+(u, x_1, y)}{\ve y \ve \varphi(\ve y \ve)} \, dy \leq \frac{C}{\uC + \oC} \left( \frac{\uC(R)}{R^2} + \oC (R) \right).
\end{align*}
Since $u(x_1) \leq \beta \leq 1$ and $u(x-y) > 0$, we have
\begin{align*}
C_\varphi \int_{\Rn} \frac{(u(x_1 + y) - 2)^+}{\ve y \ve \varphi(\ve y \ve)} \, dy \leq \frac{C}{\uC + \oC} \left( \frac{\uC(R)}{R^2} + \oC(R) \right).
\end{align*}
If $u(x_0) \leq 2$, then $\alpha = u(x_0) d^\gamma \leq 2$, which gives a uniform bound for $\alpha$. Assume that $u(x_0) > 2$, then we can estimate the second term of \eqref{e:Mw} for $x \in B_{\theta r/ 2}(x_0)$ as follows:
\begin{align*}
&\int_{B^c_{\theta r}(x_0 - x)} \frac{\left( u(x+y) - (1-\frac{\theta}{2})^{-\gamma}u(x_0) \right)^+}{\ve y \ve^n \varphi(\ve y \ve)} \, dy \\
&\leq \int_{B^c_{\theta r}(x_0 - x)} \frac{\left( u(x+y) - 2 \right)^+}{\ve y \ve^n \varphi(\ve y \ve)} \, dy \\
&= \int_{B^c_{\theta r}(x_0 - x)} \frac{\left( u(x_1 + x + y - x_1) - 2 \right)^+}{\ve x + y - x_1 \ve^n \varphi( \ve x+y-x_1 \ve)} \frac{\ve x + y - x_1 \ve^n \varphi( \ve x+y-x_1 \ve)}{\ve y \ve^n \varphi(\ve y \ve)} \, dy.
\end{align*}
We see that for $x \in B_{\theta r/2}(x_0)$ and $y \in \R \setminus B_{\theta r}(x_0 - x)$, 
\begin{align*}
\frac{\ve x+y-x_1 \ve^n}{\ve y \ve^n} \leq \left( \frac{\ve x-x_0 \ve + \ve x_0 \ve + \ve x_1 \ve}{\ve y \ve} + 1\right)^n \leq \left( \frac{CR}{\theta r} + 1 \right)^n \leq C \left( \frac{R}{\theta r} \right)^n
\end{align*}
and
\begin{align*}
\frac{\varphi( \ve x+y-x_1 \ve)}{\varphi(\ve y \ve)} \leq \frac{\varphi(CR + \ve y \ve)}{\varphi(\ve y \ve)} \leq a \left( \frac{R}{\theta r} \right)^{\os} \leq a \left( \frac{R}{\theta r} \right)^2.
\end{align*}
Therefore, we obtain
\begin{align*}
\mathcal{M}_{\mathcal{L}_0}^- w(x) 
&\leq \frac{\uC(R)}{R^2(\uC + \oC)} + \frac{C}{\uC + \oC} \left( \frac{R}{\theta r} \right)^{n+2} \left( \frac{\uC (R)}{R^2} + \oC(R) \right) \\
&\leq \frac{C}{\uC + \oC} \left( \frac{R}{\theta r} \right)^{n+2} \left( \frac{\uC(R)}{R^2} + \oC(R) \right).
\end{align*}
Now we apply the weak Harnack inequality to $w$ in $B_{\theta r/2}(x_0)$ to obtain that
\begin{align*}
&\lv \lb u < \frac{u(x_0)}{2} \rb \cap B_\frac{\theta r}{4}(x_0) \rv = \lv \lb w > \left( \left( 1-\frac{\theta}{2} \right)^{-\gamma} - \frac{1}{2} \right) u(x_0) \rb \cap B_\frac{\theta r}{4}(x_0) \rv \\
&\leq C (\theta r)^n \left( w(x_0) + C \left( \frac{R}{\theta r} \right)^{n+2} \left( 1 + \frac{R^2 \oC (R)}{\uC (R)} \right) \right)^\varepsilon \left( \left( 1-\frac{\theta}{2} \right)^{-\gamma} - \frac{1}{2} \right)^{-\varepsilon} \frac{1}{u(x_0)^\varepsilon}. 
\end{align*}
By inequalities \eqref{e:wsc 0} and \eqref{e:wsc infty}, we have
\begin{align*}
\frac{R^2 \oC (R)}{\uC (R)} \leq R^2 \left( \frac{a}{\us \varphi(R)} \right) \left( \frac{R^2}{a (2-\us) \varphi(R)} \right)^{-1} = a^2 \frac{2-\us}{\us} \leq \frac{2a^2}{\sigma_0}.
\end{align*}
Thus, we have
\begin{align*}
&\lv \lb u < \frac{u(x_0)}{2} \rb \cap B_\frac{\theta r}{4}(x_0) \rv \\
&\leq C (\theta r)^n \left( \frac{\left( \left( 1- \frac{\theta}{2} \right)^{-\gamma} - 1 \right) u(x_0) + C \left( \frac{R}{\theta r} \right)^{n+2}}{u(x_0)} \right)^\varepsilon \\
&\leq C(\theta r)^n \left( \left( \left( 1- \frac{\theta}{2} \right)^{-\gamma} - 1 \right)^\varepsilon + C \left( \frac{R}{\theta r} \right)^{(n+2) \varepsilon} \frac{1}{u(x_0)^\varepsilon} \right).
\end{align*}
Since $u(x_0) = \alpha (R/2r)^\gamma$ and $\gamma = (n+2)/\varepsilon$, we have
\begin{align*}
&\lv \lb u < \frac{u(x_0)}{2} \rb \cap B_\frac{\theta r}{4}(x_0) \rv \\
&\leq C(\theta r)^n \left( \left( \left( 1- \frac{\theta}{2} \right)^{-\gamma} - 1 \right)^\varepsilon + C \left( \frac{R}{\theta r} \right)^{(n+2 - \gamma) \varepsilon} \alpha^{-\varepsilon} \theta^{-\gamma \varepsilon} \right) \\
&\leq C(\theta r)^n \left( \left( \left( 1- \frac{\theta}{2} \right)^{-\gamma} - 1 \right)^\varepsilon + C \alpha^{-\varepsilon} \theta^{-\gamma \varepsilon} \right).
\end{align*}
We choose a uniform constant $\theta > 0$ sufficiently small so that
\begin{align*}
C(\theta r)^n \left( \left( 1 - \frac{\theta}{2} \right)^{-\gamma} - 1 \right)^\varepsilon \leq \frac{1}{4} \lv B_\frac{\theta r}{4}(x_0) \rv.
\end{align*}
If $\alpha > 0$ is sufficiently large, then we have
\begin{align*}
C(\theta r)^n \theta^{-\gamma \varepsilon} \alpha^{-\varepsilon} \leq \frac{1}{4} \lv B_\frac{\theta r}{4}(x_0) \rv,
\end{align*}
which implies that 
\begin{align*}
\lv \lb u < \frac{u(x_0)}{2} \rb \cap B_\frac{\theta r}{4}(x_0) \rv \leq \frac{1}{2} \lv B_\frac{\theta r}{4}(x_0) \rv.
\end{align*}
Therefore, $\alpha$ is uniformly bounded and the result follows.
\end{proof}

\subsection{H\"older Continuity} \label{s:4.5 Holder}

Theorem \ref{t:Holder} follows from the following lemma by a simple scaling.

\begin{lemma}
Assume that $\us \geq \sigma_0 > 0$. There exists a uniform constant $\varepsilon_0 > 0$, depending only on $n, \lambda, \Lambda, a$, and $\sigma_0$, such that if $-\frac{1}{2} \leq u \leq \frac{1}{2}$ in $\Rn$ and
\begin{align*}
\mathcal{M}^+_{\mathcal{L}_0} u \geq - \varepsilon_0 \frac{\uC (R)}{R^2(\uC + \oC)} \quad \text{and} \quad \mathcal{M}^-_{\mathcal{L}_0} u \leq \varepsilon_0 \frac{\uC (R)}{R^2(\uC + \oC)} \quad \text{in} ~ B_{R},
\end{align*}
then $u \in C^\alpha(B_{R/2})$ and
\begin{align*}
\ve u(x) - u(0) \ve \leq CR^{-\alpha} \ve x \ve^\alpha
\end{align*}
for some uniform constants $\alpha > 0$ and $C > 0$.
\end{lemma}

\begin{proof}
We will show that there exist an increasing sequence $\lb m_k \rb_{k \geq 0}$ and a decreasing sequence $\lb M_k \rb_{k \geq 0}$ satisfying $m_k \leq u \leq M_k$ in $B_{4^{-k}R}$ and $M_k - m_k = 4^{-\alpha k}$, so that the theorem holds.

For $k = 0$ we choose $m_0 = -\frac{1}{2}$ and $M_0 = \frac{1}{2}$. Now assume that we have sequences up to $m_k$ and $M_k$. We want to show that we can continue the sequences by finding $m_{k+1}$ and $M_{k+1}$.

In the ball $B_{4^{-(k+1)}R}$, either $u \geq (M_k + m_k)/2$ in at least half of the points in measure, or $u \leq (M_k + m_k)/2$ in at least half of the points. Let us say that
\begin{align*}
\lv \lb u \geq \frac{M_k + m_k}{2} \rb \cap B_{4^{-(k+1)}R} \rv \geq \frac{\ve B_{4^{-(k+1)}R} \ve}{2}.
\end{align*}
Define the function
\begin{align*}
v(x) := \frac{u(x) - m_k}{(M_k - m_k)/2}.
\end{align*}
Then $v \geq 0$ in $B_{4^{-k}R}$ by the induction hypothesis, and $\lv \lb v \geq 1 \rb \cap B_{4^{-(k+1)}R} \rv \geq \ve B_{4^{-(k+1)}R} \ve /2$. To apply Theorem \ref{t:weakHI}, we define $w = v^+$. Note that we still have $\lv \lb w \geq 1 \rb \cap B_{4^{-(k+1)}R} \rv \geq \ve B_{4^{-(k+1)}R} \ve /2$. Since $\mathcal{M}^-_{\mathcal{L}_0} u \leq \varepsilon_0 \frac{\uC (R)}{R^2(\uC + \oC)}$ in $B_{R}$, 
\begin{align*}
\mathcal{M}^-_{\mathcal{L}_0} w \leq \mathcal{M}^-_{\mathcal{L}_0} v + \mathcal{M}^+_{\mathcal{L}_0} v^- \leq \frac{\varepsilon_0}{(M_k - m_k)/2} \frac{\uC(R)}{R^2(\uC + \oC)} + \mathcal{M}^+_{\mathcal{L}_0} v^- \quad \text{in} ~ B_R.
\end{align*}
To estimate $\mathcal{M}^+_{\mathcal{L}_0} v^-$, we claim that $v(x) \geq -2\left( \frac{\ve 4x \ve^\alpha}{(4^{-k} R)^\alpha} - 1 \right)$ in $B_{4^{-k}R}$. Indeed, for $x \in B_{4^{-k+j}R} \setminus B_{4^{-k+j-1}R}, 1\leq j \leq k$, 
\begin{align*}
v(x) 
&= \frac{u(x) - m_k}{(M_k - m_k)/2} \geq \frac{m_{k-j} - M_{k-j} + M_k - m_k}{(M_k - m_k)/2} \\
&= -2(4^{\alpha j} - 1) \geq -2 \left( \frac{\ve 4x \ve^\alpha}{(4^{-k} R)^\alpha} - 1 \right),
\end{align*}
and for $x \in B_R^c$, 
\begin{align*}
v(x) 
&\geq \frac{-\frac{1}{2} - M_k + M_k - m_k}{(M_k - m_k)/2} = - (1+2M_k) 4^{\alpha k} + 2 \\
&\geq - 2( 4^{\alpha k} - 1) \geq -2 \left( \frac{\ve 4x \ve^\alpha}{(4^{-k} R)^\alpha} - 1 \right).
\end{align*}
Thus, we have for $x \in B_{3\cdot 4^{-(k+1)}R}$
\begin{align*}
\mathcal{M}^+_{\mathcal{L}_0} v^- (x) 
&\leq 2\Lambda C_\varphi \int_{\Rn} \frac{v^-(x+y)}{\ve y \ve^n \varphi(\ve y \ve)} \, dy = 2\Lambda C_\varphi \int_{v(x+y) < 0} \frac{v^-(x+y)}{\ve y \ve^n \varphi(\ve y \ve)} \, dy \\
&\leq 4\Lambda C_\varphi \int_{x+y \not\in B_{4^{-k}R}} \left( \frac{\ve 4(x+y) \ve^\alpha}{(4^{-k} R)^\alpha} - 1 \right) \frac{dy}{\ve y \ve^n \varphi(\ve y \ve)}.
\end{align*}
If $x \in B_{3\cdot 4^{-(k+1)}R}$ and $x+y \not \in B_{4^{-k}R}$, then $\ve y \ve \geq \ve x+y \ve - \ve x \ve > 4^{-k}R - 3\cdot 4^{-(k+1)}R = 4^{-(k+1)}R$ and $\ve x+y \ve \leq \ve x \ve + \ve y \ve \leq 3 \cdot 4^{-(k+1)}R + \ve y \ve \leq 4\ve y \ve$. Thus we obtain
\begin{align*}
\mathcal{M}^+_{\mathcal{L}_0} v^-(x) 
&\leq 4\Lambda C_\varphi \int_{\ve y \ve > 4^{-(k+1)}R} \left( \left( \frac{16\ve y \ve}{4^{-k}R} \right)^\alpha -1 \right) \frac{dy}{\ve y \ve^n \varphi(\ve y \ve)} \\
&= C C_\varphi \int_{4^{-(k+1)}R}^\infty \left( \left( \frac{16r}{4^{-k}R} \right)^\alpha - 1 \right) \frac{dr}{r \varphi(r)} \\
&\leq CC_\varphi a \frac{(4^{-(k+1)}R)^{\us}}{\varphi(4^{-(k+1)}R)} \int_{4^{-(k+1)}R}^\infty \left( \left( \frac{16r}{4^{-k}R} \right)^\alpha - 1 \right) r^{-1-\us} \, dr.
\end{align*}
If we have taken $\alpha < \sigma_0$, then
\begin{align*}
\int_{4^{-(k+1)}R}^\infty \left( \left( \frac{16r}{4^{-k}R} \right)^\alpha - 1 \right) r^{-1-\us} \, dr = \left( \frac{4^\alpha}{\us - \alpha} - \frac{1}{\us} \right) (4^{-(k+1)}R)^{-\us}.
\end{align*}
Since
\begin{align*}
\frac{4^\alpha}{\us - \alpha} - \frac{1}{\us} = \frac{\us (4^\alpha - 1) + \alpha}{\us(\us - \alpha)} \leq \frac{2(4^\alpha - 1) + \alpha}{\sigma_0 (\sigma_0 - \alpha)} =: f(\alpha, \sigma_0),
\end{align*}
we have
\begin{align*}
\mathcal{M}^+_{\mathcal{L}_0} v^-(x) \leq \frac{C C_\varphi a}{\varphi(4^{-(k+1)}R)} f(\alpha, \sigma_0).
\end{align*}
Hence,
\begin{align*}
\mathcal{M}^-_{\mathcal{L}_0} w \leq 2\varepsilon_0 4^{\alpha k} \frac{\uC(R)}{R^2(\uC + \oC)} + \frac{C C_\varphi a}{\varphi(4^{-(k+1)}R)} f
\end{align*}
in $B_{3 \cdot 4^{-(k+1)} R}$. Note that the same holds in $B_{4^{-(k+1)}R}(x)$ for $x \in B_{\frac{1}{2}4^{-k}R}$. We applying Theorem \ref{t:weakHI} to $w$ in $B_{4^{-(k+1)}R}(x)$ to obtain
\begin{align*}
&\frac{\ve B_{4^{-(k+1)}R} \ve}{2} \leq \lv \lb w \geq 1 \rb \cap B_{4^{-(k+1)}R} \rv \\
&\leq C \left( \frac{R}{4^{k+1}} \right)^n \left( w(x) + \left( 2\varepsilon_0 4^{\alpha k} \frac{\uC(R)}{R^2(\uC + \oC)} + \frac{C C_\varphi af}{\varphi(\frac{R}{4^{k+1}})} \right) \left( \frac{R}{4^{k+1}} \right)^2 \frac{\uC + \oC}{\uC(\frac{R}{4^{k+1}})} \right)^\varepsilon \\
&\leq C \left( \frac{R}{4^{k+1}} \right)^n \left( w(x) + \frac{\varepsilon_0}{8} 4^{(\alpha -2)k} \frac{\uC (R)}{\uC(4^{-(k+1)} R)} + C\frac{(4^{-(k+1)}R)^2 f}{\varphi(4^{-(k+1)} R) \uC(4^{-(k+1)} R)} \right)^\varepsilon.
\end{align*}
We have
\begin{align*}
\frac{\uC(R)}{\uC(4^{-(k+1)} R)} \leq 1 + a^2 4^{(2-\us)(k+1)} \leq 1 + 16 a^2 4^{(2-\us)k}
\end{align*}
by the inequality \eqref{e:wsc ratio}, and
\begin{align*}
\frac{(4^{-(k+1)}R)^2}{\varphi(4^{-(k+1)} R) \uC(4^{-(k+1)} R)} \leq a(2-\us) \leq 2a
\end{align*}
by the inequality \eqref{e:wsc 0}. Therefore, using $\alpha < \sigma_0$, we have
\begin{align*}
\theta
&\leq w(x) + \frac{\varepsilon_0}{8} 4^{(\alpha-2)k} (1+ 16 a^2 4^{(2-\us)k}) + Ca f(\alpha, \sigma_0) \\
&= w(x) + \frac{\varepsilon_0}{8} 4^{(\alpha-2)k} + 2\varepsilon_0 a^2 4^{(\alpha-\us)k} + Ca f(\alpha, \sigma_0) \\
&\leq w(x) + \frac{\varepsilon_0}{8} + 2\varepsilon_0 a^2 + Ca f(\alpha, \sigma_0),
\end{align*}
where $\theta > 0$ is a uniform constant. Notice that we have $\lim_{\alpha \rightarrow 0^+} f(\alpha, \sigma_0) = 0$. If we have chosen $\alpha$ and $\varepsilon_0$ satisfying
\begin{align*}
g(\alpha) < \frac{\theta}{4Ca^2}, \quad 4^{-\alpha} \geq 1- \theta/4, \quad \text{and} \quad \varepsilon_0 < \frac{\theta}{4}\left( \frac{1}{8} + 2a^2 \right),
\end{align*} 
then we have $w \geq \theta/2$ in $B_{\frac{1}{2}4^{-k} R}$. Thus, if we let $M_{k+1} = M_k$ and $m_{k+1} = M_k - 4^{-\alpha(k+1)}$, then 
\begin{align*}
M_{k+1} \geq u \geq m_k + \frac{M_k - m_k}{4} \theta = M_k - \left(1 - \frac{\theta}{4} \right) 4^{-\alpha k} \geq M_k - 4^{-\alpha(k+1)} = m_{k+1}
\end{align*}
in $B_{4^{-(k+1)}R}$.

On the other hand, if $\lv \lb u \geq (M_k + m_k)/2 \rb \cap B_{4^{-(k+1)}R} \rv \geq \ve B_{4^{-(k+1)}R} \ve / 2$, we define
\begin{align*}
v(x) = \frac{M_k - u(x)}{(M_k - m_k)/2}
\end{align*}
and continue in the same way using that $\mathcal{M}^+_{\mathcal{L}_0} u \geq - \varepsilon_0 \frac{\uC(R)}{R^2(\uC + \oC)}$.
\end{proof}

\section*{Acknowledgement}

The research of Minhyun Kim is supported by
 the National Research Foundation of Korea (NRF) grant funded by the Korea government (MSIP) : NRF-2016K2A9A2A13003815. 
The research of Ki-Ahm Lee is supported by the National Research Foundation of Korea(NRF) grant funded by the Korea government(MSIP) : NRF-2015R1A4A1041675.

\end{document}